\documentclass{ociamthesis}
\usepackage[english]{babel}

\pdfoutput=1

\usepackage{amsmath}
\usepackage{amssymb}
\usepackage{amsthm}
\usepackage{amsfonts}
\usepackage{graphicx}
\usepackage[colorlinks=true, allcolors=blue]{hyperref}
\usepackage{caption}
\usepackage{subcaption}
\usepackage{float}
\usepackage{fancyhdr}
\usepackage{mcode}
\usepackage[margin=50pt]{geometry}
\usepackage{graphbox}

\newtheorem{theorem}{Theorem}[section]
\newtheorem{lemma}[theorem]{Lemma}
\newtheorem{definition}[theorem]{Definition}
\newtheorem{remark}[theorem]{Remark}
\newtheorem*{convention}{Convention}

\begin{document}

\setlength{\headheight}{15pt}
\setcounter{secnumdepth}{3}
\setcounter{tocdepth}{3}

\pagestyle{fancy}
\fancyhf{}
\rhead{Henry Ginn}
\lhead{\nouppercase{\leftmark}}
\cfoot{\thepage}

\def\logo{{\includegraphics[width=64mm]{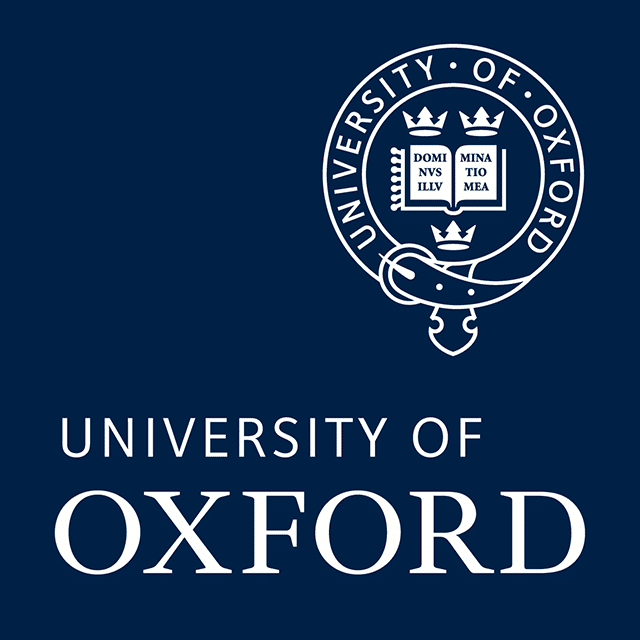}}}
\title{Lightning Helmholtz Solver}
\author{Henry Ginn}
\college{Exeter College}
\supervisor{Lloyd N. Trefethen}
\degree{Master of Mathematics}
\degreedate{Trinity 2022}

\maketitle

\begin{abstract}

In this dissertation we have applied Trefethen and Gopal's Lightning Method to solve the Helmholtz equation in the exterior of two dimensional piecewise smooth domains. The background theory motivating the method is presented, and we explore the optimal method implementation for the unit square, which is subsequently used to give a guide on parameter selection for a general region. The behaviour of the computed solutions is verified to act in accordance with our intuition and current understanding of wave propagation, and we show that the wave decays to approximately 0 in the shadow region.

\end{abstract}

\newpage
\tableofcontents
\newpage
\section{Introduction}

In the modern world, there seems to be an ever-growing demand upon mathematicians to deliver increasingly accurate solutions to increasingly difficult problems. Analytic solutions to such problems are a rare breed, and we find ourselves resorting to numerical methods to have any hope of attaining a satisfactory answer. In this dissertation, we aim to expand upon the \textit{the Lightning Method}, a promising new numerical method developed by Trefethen and Gopal \cite{doi:10.1073/pnas.1904139116}, looking at the particular case of the Helmholtz problem.

The method tackles domains with piecewise smooth boundaries, and even works with non convex domains. It is able to handle a wide variety of regions and boundary data, and with the help of our code it is very simple to use, only requiring basic knowledge when solving unambitious problems. For harder problems, some specific implementation choices must be selected to attain a useful output, and this is one of the areas we will explore. Given suitable choices, we have typically achieved 3-10 digits of accuracy, often in a few seconds, which is more than sufficient to produce good plots. Unfortunately we are limited by the machine precision of a standard computer, and the extremely poor conditioning of the matrices involved, so this method is not suited to high accuracy solutions. MATLAB's \cite{MATLAB} sluggish Hankel function implementation also means plots with a large number of evaluations can take an unsatisfactory amount of time to compute, especially for harder problems that can require significantly more evaluations. Whilst this method almost certainly won't become the new industry standard method, it certainly deserves its place in medium level black box solvers, and is an excellent way of finding an initial picture of the solution.

Trefethen and Gopal have applied the Lightning Method to the Laplace problem with great success \cite{doi:10.1137/19M125947X}, and provided some theoretical support for the convergence of the method as well. The motivation for the method comes from a result due to Newman, where rational approximation is used on $f(x) = |x|$ to achieve root exponential convergence in the degree of the rational functions used \cite{Newman1964RationalAT}. Trefethen and Gopal extend this idea to a region with corners using fundamental solutions to a partial differential equation, and perhaps surprisingly, achieve the same root exponential convergence globally over the entire region. In the case of the Laplace problem, these solutions are inverse linear functions, $\frac{1}{z-z_j}$, where the $z_j$ are specially chosen a priori. The Lightning Method prescribes how these $z_j$ are chosen, and because these degrees of freedom are fixed in advance, our problem ends up reducing to a least squares problem. This is a significantly simpler problem than if the poles were not already chosen, which is why the Lightning Method is so fast for what it does.

We give a brief overview of the algorithm and our implementation. For more detail, see section 3 of Trefethen and Gopal's paper for the Laplace problem \cite{doi:10.1137/19M125947X}, or the MATLAB \cite{MATLAB} code found in the appendix.
\begin{enumerate}
    \item The problem is defined
    \begin{enumerate}
        \item The boundary is defined in the complex plane by $m$ corners, $z^k$, $k = 1,\cdots, m$, and functions $f_k:[0, 1] \to \mathbb{C}$ with $f_k(0) = z^{k-1}$ and $f_k(1) = z^k$ where we define $z^0 = z^m$
        \item The boundary data and wave number are defined.
    \end{enumerate}
    \item Points of the method are computed
    \begin{enumerate}
        \item The locations of the poles are set, distributed along the interior bisectors with exponential clustering at the corners.
        \item The locations of the sample points are set, distributed along the edges with approximately exponential clustering at the corners.
    \end{enumerate}
    \newpage
    \item The problem is solved
    \begin{enumerate}
        \item We form the problem matrix, $A$, which encodes the assumed form of solution, the sample points, and the pole locations.
        \item We form the problem data vector, $b$, by evaluating the boundary data at our sample points.
        \item The vector of coefficients of the terms in our form of solution, $x$, is found via solving the least squares problem $\min \limits_{x} ||Ax-b||^2$
        \item The solution is given by the assumed form of solution with coefficients given by the vector $x$.
    \end{enumerate}
\end{enumerate}

Unlike Trefethen and Gopal, we have not specified here roughly what values the parameters in the above should take. Our code has preset default values for all choices that one could make, but these are intended to be changed. It is written to be a tool for exploring the Lightning Method rather than simply as a solver, and best (or in some cases, any) performance is attained when the right choices for the problem are used.

We start off section 2 with the necessary prerequisites we will need in this dissertation. Then we present the motivation for studying the Helmholtz equation, and explain why piecewise smooth regions with corners should be given attention. Section 3 builds up the theory behind the method, and sheds light on the possible forms of solution we can use. Then in section 4, we illustrate and motivate the journey taken if one were to do a deep dive into optimising the method performance for a particular problem. We give a summary of how to diagnose performance issues, and prescribe how these can be alleviated in the discussion. Our final section is dedicated to testing the algorithm over a range of wave phenomena, and verifying that our solutions act as we expect them to.

Before we begin with the background theory, we present a few examples of the method in action.

\begin{figure}[H]
     \centering
     \begin{subfigure}[t]{0.32\textwidth}
         \centering
         \includegraphics[width=1.2\textwidth]{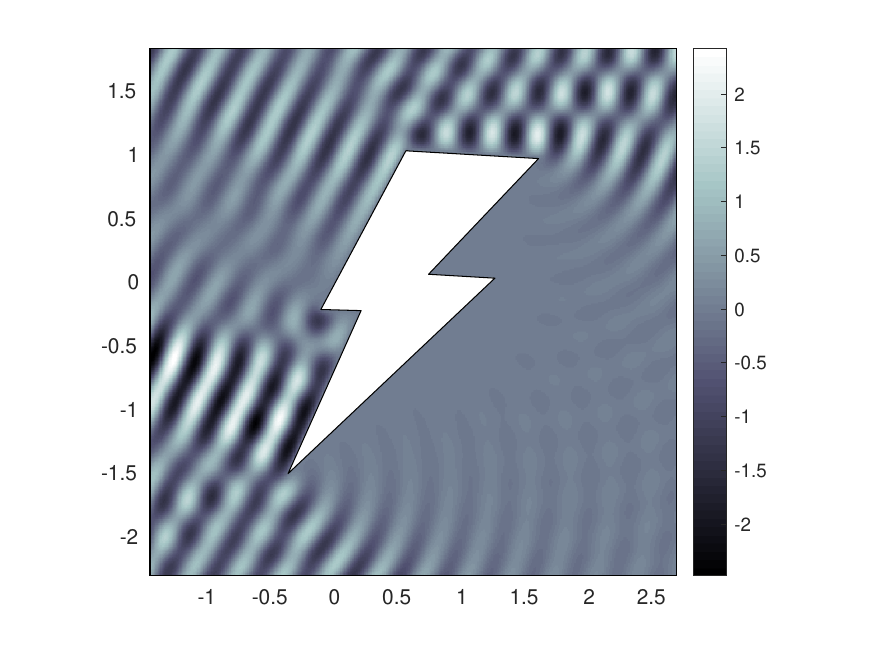}
     \end{subfigure}
     \hfill
     \begin{subfigure}[t]{0.32\textwidth}
         \centering
         \includegraphics[width=1.2\textwidth]{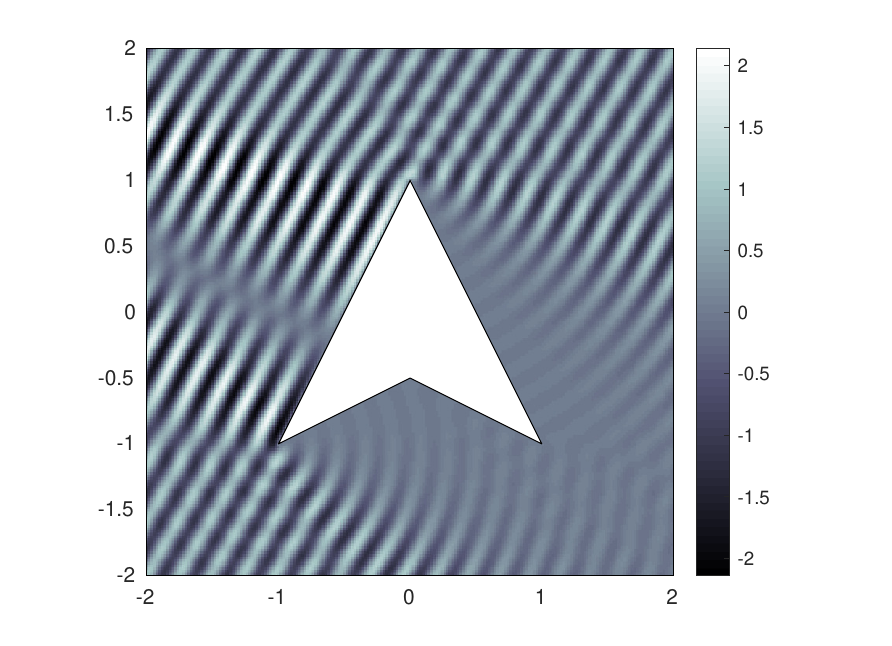}
     \end{subfigure}
     \begin{subfigure}[t]{0.32\textwidth}
         \centering
         \includegraphics[width=1.2\textwidth]{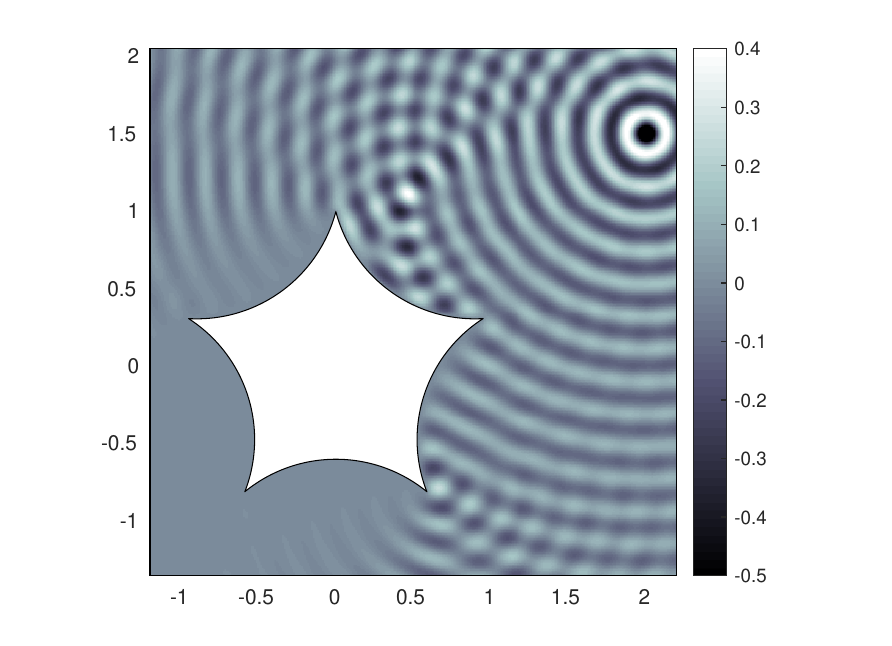}
         \centering
     \end{subfigure}
     \hfill
     \begin{subfigure}[t]{0.32\textwidth}
         \centering
         \includegraphics[width=1.2\textwidth]{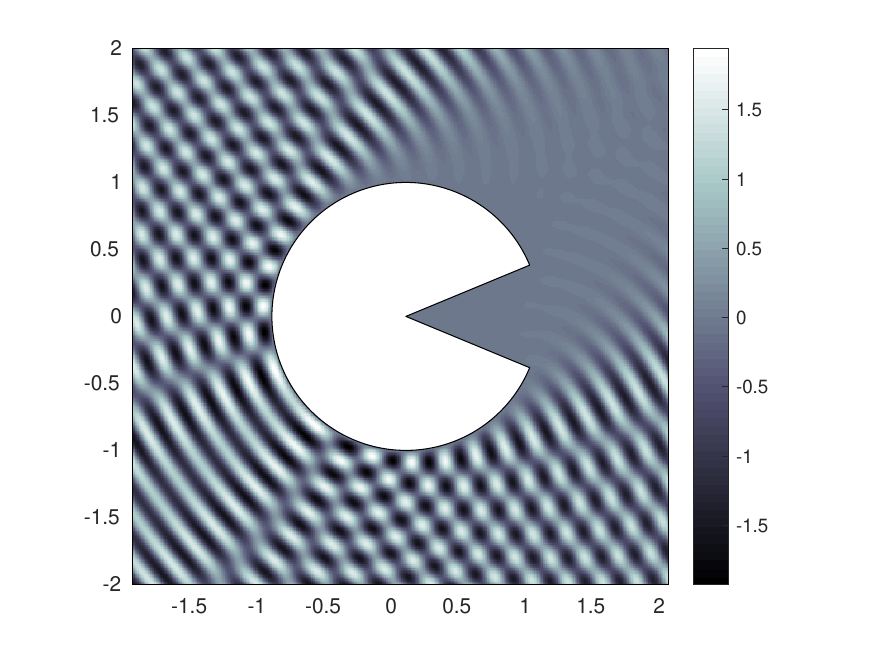}
     \end{subfigure}
     \begin{subfigure}[t]{0.32\textwidth}
         \centering
         \includegraphics[width=1.2\textwidth]{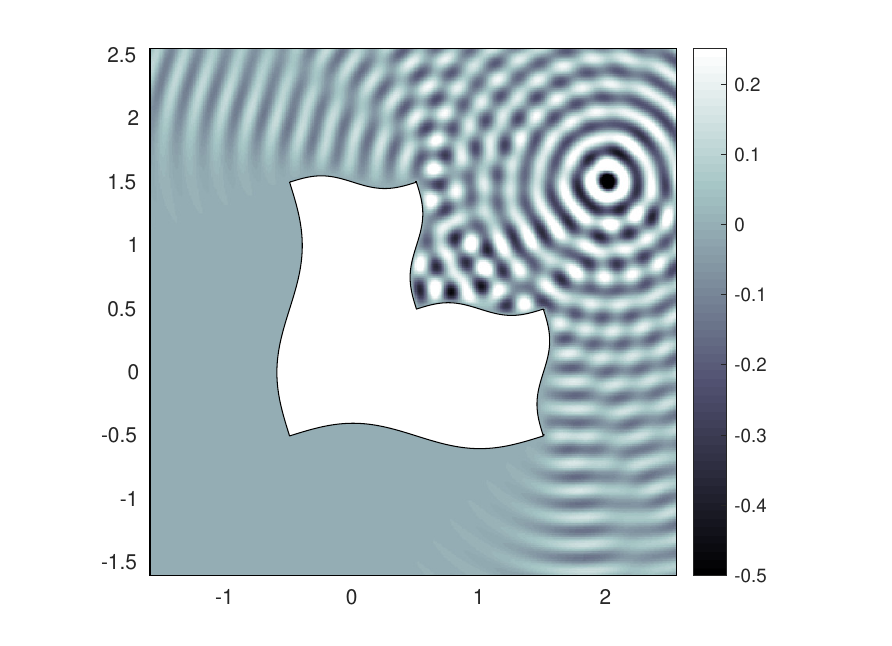}
     \end{subfigure}
        \caption{A selection of examples solutions to the Helmholtz equation around regions with corners}
\end{figure}

\section{The Helmholtz Problem}

\subsection{Notation and Prerequisites}

We start off by introducing some notation, definitions, and conventions.

\begin{convention}
We represent $(x,y) \in \mathbb{R}^2$ using real and imaginary components of a complex variable.
\begin{center}
    $x = \operatorname{Re}(z)$, $y = \operatorname{Im}(z)$, $z \in \mathbb{C}$.
\end{center}
\end{convention}

\begin{definition}[Helmholtz Equation]
The Helmholtz equation is given by the following PDE:
$$\nabla^2 u+k^2 u = 0.$$
\end{definition}

\begin{remark}
This is taken to hold in some region and usually accompanied by Dirichlet, Neumann, or mixed boundary data. We will restrict our focus to Dirichlet boundary data.
\end{remark}

\begin{definition}[Bessel Functions]
The Bessel functions of the first and second kind are linearly independent solutions to the Bessel equation,

$$z^2\frac{\partial^2u}{\partial z^2} + z\frac{\partial u}{\partial z} + (z^2-\alpha^2) = 0$$

and are given respectively by the following \cite{NIST:DLMF}:
$$J_\alpha(z) = \left(\frac{z}{2}\right)^\alpha \sum_{k=0}^{\infty}(-1)^k\frac{\left(\frac{z}{2}\right)^{2k}}{k!\Gamma(\alpha+k+1)},$$
$$Y_\alpha(z) = \frac{J_\alpha(z)\cos(\alpha \pi) - J_{-\alpha}(z)}{\sin(\alpha \pi)},$$
where the limiting value is taken for $Y_\alpha$ when $\alpha$ is an integer. 
\end{definition}

\begin{remark}
These are entire functions of $\alpha$ for fixed $z \ne 0$, although we will only require integer values, for which they simplify down to,
$$J_n = \left(\frac{z}{2}\right)^n \sum_{k=0}^{\infty}(-1)^k\frac{\left(\frac{z}{2}\right)^{2k}}{k!(n+k)!},$$
$$Y_n(z) = \frac{1}{\pi}\frac{\partial J_\alpha(z)}{\partial \alpha}\,\Bigr\rvert_{\alpha = n} + \frac{(-1)^n}{\pi}\frac{\partial J_\alpha(z)}{\partial \alpha}\,\Bigr\rvert_{\alpha=-n}.$$
\end{remark}

\begin{definition}[Hankel Functions]
The Hankel functions of the first and second kind are defined by \cite{NIST:DLMF}
\begin{align*}
    &H_n^{(1)}(z) = J_n(z)+iY_n(z), &H_n^{(2)}(z) = J_n(z)-iY_n(z).
\end{align*}
\end{definition}

\begin{remark}
These also solutions to the Bessel equation by linearity, and are sometimes called the Bessel functions of the third kind. The motivation behind why these particular linear combinations deserves their own notation will be explained later by the Sommerfeld radiation condition.
\end{remark}

\begin{definition}[Dirac Delta Function]
The Dirac delta function is a generalised function satisfying the following properties:
\begin{align*}
    & \delta(x) = 0 \text{ for } x \ne 0,
    & \int_{-\infty}^{\infty} \delta(x) dx = 1.
\end{align*}
\end{definition}

\begin{remark}
Despite the name, no such classical function exists, and it should be treated as a distribution via integrals. It can be thought of as having ``all its mass" centred at the origin.
\end{remark}

\begin{convention}
When looking for time harmonic solutions, we use $e^{-i\omega t}$ as the time evolution component. For a further explanation on this, see remark \ref{time harmonic convention}
\end{convention}

\subsection{The Helmholtz Equation and Applications}
The precise problem of interest that we wish to solve is as follows:
$$
\begin{cases}
    \nabla^2 u+k^2 u = 0 & \text{ in } \mathbb{R}^2 \setminus \Omega\\
    \hfill u = f & \text{ on } \delta \Omega
\end{cases}
$$
where $\Omega$ is a simply connected region with piecewise smooth boundary. The exact meaning of smoothness is beyond the scope of this dissertation, although all our components will be analytic.

This equation is often reached when solving the wave equation, \hbox{$\frac{\partial^2u}{\partial t^2}=c^2\nabla^2u$}, where we look for separable solutions of the form $u(x, y, t) = R(x, y)T(t)$.

\begin{align}
\label{wave equation}
    \begin{split}
    \frac{\partial^2u}{\partial t^2} &= c^2\nabla^2u\\
    \frac{\partial^2}{\partial t^2}R(x,y)T(t) &= c^2\nabla^2R(x,y)T(t) \\
    R(x,y)T''(t) &= c^2T(t)\nabla^2R(x,y) \\
    \frac{T''(t)}{T(t)} &= c^2\frac{\nabla^2R(x,y)}{R(x,y)} = -\omega^2 \\
    \nabla^2R &= -\frac{\omega^2}{c^2}R \\
    \nabla^2R + k^2R &= 0,
    \end{split}
\end{align}

Here we have introduced $\omega$, as each side of the equation must be independent of $t$, $x$, and $y$ (this turns out to be the frequency of the wave), and the wave number, $k=\frac{\omega}{c}$. When we come to plot solutions to the Helmholtz equation, we will see such wave-like behaviour forming. This link to waves gives rise to a wealth of applications, such as electromagnetism, quantum theory, and optics. It is also a commonly-used method for solving hyperbolic PDEs, where an integral transform is used to convert the PDE into the Helmholtz equation, making it easier to solve \cite{enwiki:1069719909}. It was also used by Euler and Lagrange in the 1700s to study vibrating membranes and sound propagation \cite{SCHOT1992385}

\subsection{Choice of Region Type}

We split the space of possible regions into three rough categories. The first of these has a very smooth boundary, and these are the easiest domains over which to solve over. There are fewer opportunities for complications to arise than the other region types we will look at, and standard methods for solving PDEs such as boundary integral equations and the method of fundamental solutions achieve very good convergence over such regions \cite{BARNETT20087003}. The second type has a pathological boundary, with awkward smoothness conditions. While these may be of interest to the pure mathematician, such study is likely not going to be best suited to many real-world scenarios where intricate regularity conditions are rarely relevant.

This leads us to the third type of region, the ones we are interested in solving. This is a middle ground between the other two types, and has a piecewise smooth boundary. In practise regions often have corners, making this region type too restrictive. Examples may include the flow of air inside a cylinder of an engine, or water flowing around a square obstruction. Due to the non-physical nature of the second type of region, we find that the third type of region is often most suited for applications.

\begin{figure}[H]
     \centering
     \begin{subfigure}[t]{0.3\textwidth}
         \centering
         \includegraphics[width=\textwidth]{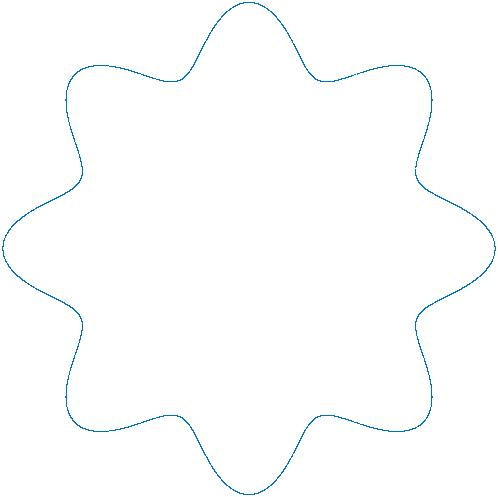}
         \caption{Type 1: smooth}
     \end{subfigure}
     \hfill
     \begin{subfigure}[t]{0.3\textwidth}
         \centering
         \includegraphics[width=\textwidth]{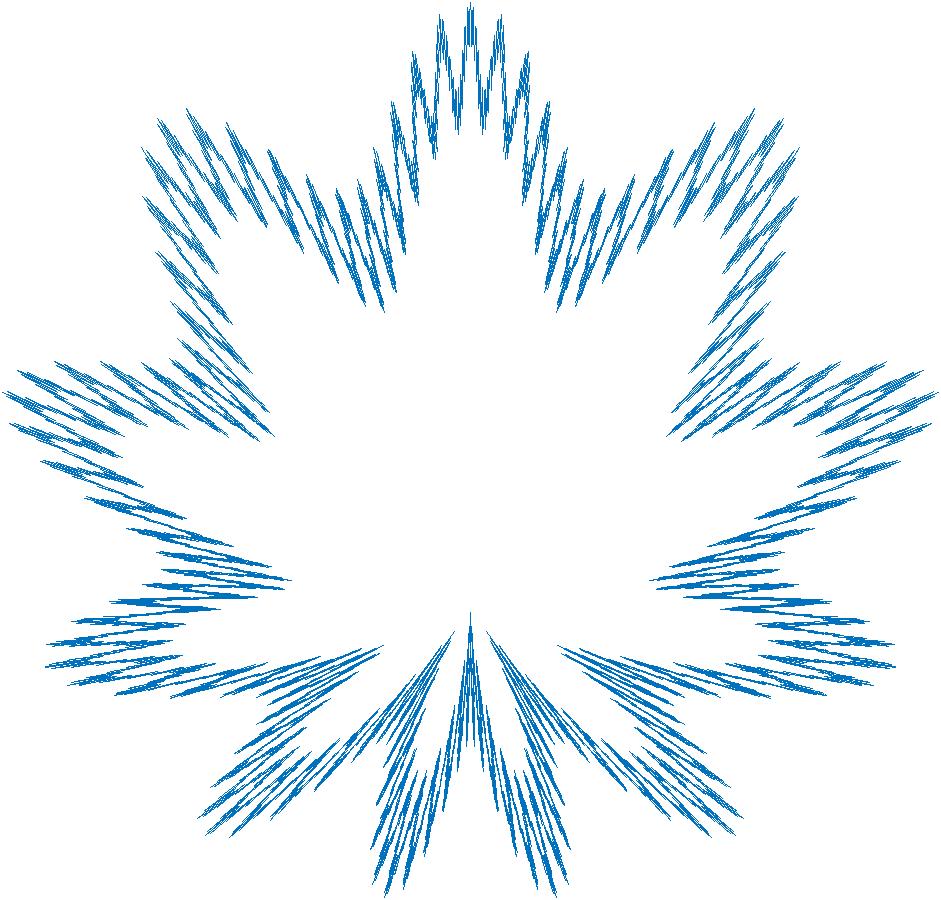}
         \caption{Type 2: not smooth}
     \end{subfigure}
     \hfill
     \begin{subfigure}[t]{0.3\textwidth}
         \centering
         \includegraphics[width=\textwidth]{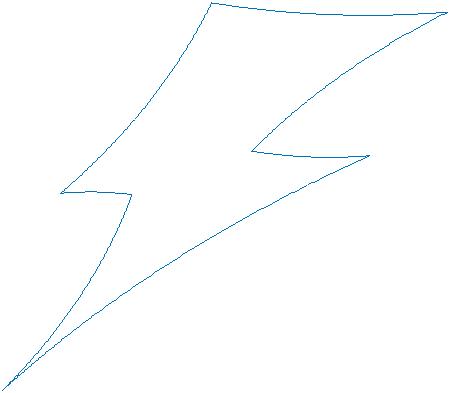}
         \caption{Type 3: piecewise smooth}
     \end{subfigure}
        \caption{The three types of region}
\end{figure}

\section{The Method of Fundamental Solutions}

\subsection{Solutions of the Helmholtz Equation}

\begin{lemma}\label{Solutions of helmholtz}
If $f_n(x)$ satisfies the Bessel equation of order $n$, then $u(z) = f_n(k \left| z \right|)\frac{z^n}{|z|^n}$ satisfies the Helmholtz equation for $z \ne 0$ and wave number $k$, that is, it solves the following equation,
$$\nabla ^2 u+k^2 u=0 \text{, for } z \ne 0$$
\end{lemma}

\begin{proof}
We change to polar coordinates, $z = r e^{i\theta}$, and consider $r \ne 0$. We also introduce the rescaling $w = kr$, so this gives ${u(r e^{i\theta}) = f_n(w)e^{in\theta}}$. 

\begin{align*}
    & \frac{\partial^2u}{\partial x^2} + \frac{\partial^2u}{\partial y^2} + k^2u = 0 \\
    \iff & \frac{\partial^2 u}{\partial r^2} +\frac{1}{r} \frac{\partial^2 u}{\partial r^2} + \frac{1}{r^2} \frac{\partial^2 u}{\partial \theta^2} + k^2u = 0 \\
    \iff & \frac{\partial^2}{\partial w^2}\left(f_n(w)e^{in\theta} \right) \frac{d^2w}{dr^2} + \frac{1}{r}\frac{\partial}{\partial w}\left( f_n(w)e^{in\theta}\right)\frac{dw}{dr} + \frac{i^2n^2}{r^2}f_n(w)e^{in\theta}+k^2f_n(w)e^{in\theta} = 0 \\
    \iff & k^2\frac{d^2f_n}{dw^2}+\frac{k \cdot k}{w} \frac{df_n}{dw} - \frac{k^2n^2}{w^2}f_n + k^2f_n = 0\\
    \iff & w^2\frac{d^2f_n}{dw^2}+w\frac{df_n}{dw}+(w^2-n^2)f_n = 0
\end{align*}
\end{proof}

We are going to use the above lemma to generate solutions to our problem, although as we have seen, the Bessel equation has multiple solutions. Thus to choose which solutions to take, we need an extra condition.

\subsection{Sommerfeld Radiation Condition}
The Sommerfeld radiation condition eliminates solutions that correspond to inward travelling waves on the motivation that these are unphysical. For a wave to be inward travelling, it would require a source point at infinity, or a sink point within a finite distance, and both of these are physically unrealised \cite{SCHOT1992385}. We also note that standing waves can be decomposed into an inward travelling and an outward travelling component, and the Sommerfeld radiation condition also cleanly rules out such possibilities by the linear nature of the Helmholtz equation.

\begin{definition}[Sommerfeld Radiation Condition \cite{Sommerfeld1912}] A solution, to the two dimensional inhomogeneous Helmholtz equation, $u(x)$, is said to be radiating if it satisfies the Sommerfeld radiation condition,
$$
\lim_{|z| \to \infty} \sqrt{|z|} \left(\frac{\partial u(z)}{\partial |z|} -iku(z) \right)=0
$$
uniformly in all directions, $\hat{z}=\frac{z}{|z|}$
\end{definition}

While we do not prove this here, we offer some intuition behind it \cite{390536}. If we take a unit complex number, $e^{i\phi}$, then $u(z)=Ae^{ik\operatorname{Re}\left(z e^{-i\phi}\right)}$ is a plane wave with constant amplitude $A$ travelling at an angle of $\phi$ from the real axis which solves the homogeneous Helmholtz equation, $\nabla^2u+k^2u=0$. Writing $z = |z|e^{i\theta}$, with $\theta$ being the angle from $e^{i\phi}$ to $z$, we have

\begin{gather*}
u(z) = A\exp\left(ik\operatorname{Re}\left(|z|e^{i(\theta-\phi)}\right)\right) = A\exp\left(ik|z|\cos{(\theta-\phi)} \right), \\
\frac{\partial u(z)}{\partial |z|} -iku(z) = (\cos(\theta-\phi)-1)iku(z).
\end{gather*}

Therefore, for this to satisfy the Sommerfeld radiation condition, we need $\theta \to \phi$ sufficiently fast as $|z| \to \infty$ for all $\phi$. In other words, in every direction, $u(z)$ needs to be sufficiently close to a plane wave travelling in that direction, or $0$, and the difference needs to decrease faster than the growth of $\sqrt{|z|}$. We see that any non trivial energy-carrying wave approaching infinity cannot be travelling inwards if it is to satisfy the Sommerfeld radiation condition.

Going back to our original question of choosing solutions to the Helmholtz equation, we note that for $z=re^{i\theta}$, $J_n(kr)e^{\pm in\theta}$ and $Y_n(kr)e^{\pm in\theta}$ do not satisfy the Sommerfeld radiation condition \cite{multipole}. This is what motivates the definition of the Hankel functions however, as they were constructed to represent outward and inward travelling waves, and thus $H_n^{(1)}(z)$ satisfies the condition. When applying lemma \ref{Solutions of helmholtz} with $f_n(z) = H_n^{(1)}(z)$, we get the following solutions:
$$
H_0^{(1)}(kr)\text{, }H_n^{(1)}(kr)e^{-in\theta}\text{, and }H_n^{(1)}(kr)e^{in\theta}\text{, }n \in \mathbb{N}
$$

\begin{remark}{\label{time harmonic convention}}
When solving the wave equation, there is a choice in the time component of the solution. The convention is to use $T(t) = e^{-i\omega t}$, although engineers use the other choice, $T(t) = e^{i\omega t}$. This results in a different sign being used in the Sommerfeld radiation condition, and the direction of the waves that the Hankel functions represent is reversed. The only difference for our purposes is that $H_n^{(2)}$ is used instead of $H_n^{(1)}.$ As stated earlier, we will follow the convention and use $T(t) = e^{-i\omega t}$.
\end{remark}

\subsection{The Method of Fundamental Solutions}
\begin{definition}[Fundamental solution]A fundamental solution to a linear partial differential equation, $\mathfrak{L}u(x)=f(x)$, is a function, $u$, that satisfies,
$$\mathfrak{L}u(x)=\delta(x),$$
where $\delta(x)$ is the Dirac delta function
\end{definition}

For the Helmholtz equation this equation is $\nabla^2 u + k^2u = \delta(z)$. $u(z) = \frac{i}{4} H_0^{(1)}(k|z|)$ satisfies this, and we see that the singular behaviour of $Y_n(x)$ as $x \to 0$ suggests that some multiple of $H_n^{(1)}(kr)e^{\pm in\theta}$ will also be a fundamental solution for each $n$, and this is indeed the case (these are anisotropic solutions) \cite{multipole}. This also gives insight into why they satisfy the Sommerfeld radiation condition - they correspond to waves emanating from a point source.

The method of fundamental solutions \cite{KUPRADZE196482} proposes that we look for a solution of the following form:

$$u(z) = \sum \alpha_i \phi_i(z),$$

where $\phi_i(z)$ are fundamental solutions to the problem being solved, and $\alpha_i$ are constants. The motivation behind this approach comes from the fact that by linearity, $u(z)$ automatically satisfies the PDE being solved, and the problem is reduced to fixing constants so that the boundary data is satisfied. We notice that we can substitute $z-z_j$ in place of $z$ for a given $z_j$, which allows us to change where the poles of our solutions are; if $u(z)$ is a fundamental solution, then $\mathfrak{L}u(z-z_j) = \delta(z-z_j)$. To avoid any unphysical behaviour from these singularities, we have to place the poles outside the domain that we are solving in. We are solving the Helmholtz problem for an exterior domain, so we will be placing the singularities in the interior of our domain.

\subsection{Form of Solution}
Here we give some motivating results and properties that lead us to assume a particular form for our solution from which the Lightning Method will be realised.

We start off with a result from Neumann \cite{Newman1964RationalAT}, where he showed that rational approximation to the function $f(x) = |x|$ could achieve root exponential convergence, that is, $||f(x)-r_n(x)|| = O( \exp(-C \sqrt{n}))$ for some $C>0$. This is in contrast to polynomial approximation which is infamously bad at approximating near singularities, only managing at best linear convergence, $||f(x)-p_n(x)|| = O\left(\frac{1}{n}\right)$. We see that polynomial approximations perform much better on smooth functions than non-smooth functions, and boast exponential, or even super-exponential convergence for analytic functions.

This prompts the use of a rational expansion to handle the corners, and then a polynomial part to get finer detail. This can be thought of as the rational part subtracting off the singular elements of the problem, leaving the smooth part to be handled by the polynomial expansion. The rational and polynomial parts are called the Newman and Runge parts respectively \cite{doi:10.1137/19M125947X}, coming from Newman's insight about using rational functions for corner singularities, and Runge's work with polynomial approximation for smooth functions. As we need our functions to satisfy the Helmholtz equation, we will work by analogy and use the fundamental solutions we mentioned above, giving the following form:

\begin{equation}\label{Form of solution extended runge}
\sum_{j=1}^{N_1} \left(a_jH_0^{(1)}(k|z_j|) + b_jH_1^{(1)}(k|z_j|) \frac{z_j}{|z_j|} \right) + \sum_{n=-N_2}^{N_2}c_nH_{|n|}(k|z_*|) \frac{z_*^n}{|z_*^n|},
\end{equation}

where $N_1$ is the size of the Newman part, $N_2$ is the size of the Runge part, $\{z_j\}$ are poles, and $z_*$ is a fixed internal point. It turns out to be the case that the negative indices in the second sum are not advantageous, and we will use the same form as Trefethen and Gopal in our analysis, given explicitly below.

\begin{equation}\label{Form of solution}
\sum_{j=1}^{N_1} \left(a_jH_0^{(1)}(k|z_j|) + b_jH_1^{(1)}(k|z_j|) \frac{z_j}{|z_j|} \right) + \sum_{n=0}^{N_2}c_nH_{|n|}(k|z_*|) \frac{z_*^n}{|z_*^n|},
\end{equation}

In general, for given $N_1$ and $N_2$, finding optimal values for $a_j$, $b_j$, $c_n$, $z_j$, and $z_*$ is a non-linear problem, and thus difficult to solve in general. This leads us to the subject of this dissertation, \emph{The Lightning Method}.

\subsection{The Lightning Method}
The key feature of the Lightning Method is that a particular distribution of poles gives us the same root exponential convergence that Newman found, and the precise placement of these poles is not as delicate as one might initially assume. This allows us to place the poles \emph{a priori}, and thus reduce the problem to an over determined linear system that can quickly be solved as a least squares problem. Such problems are solved routinely, and are far easier than the problem we started out with. The key property of the distribution of the poles is that they are exponentially clustered at the corners. Put explicitly, for a sequence of $n$ poles, $z_j^k$, associated with a particular corner, $z^k$, we have $|z_j^k-z^k| = Ae^{-\frac{Bj}{\sqrt{n}}}$ for some constant method parameters, $A$ and $B$, with $j=0,1, \cdots, n-1$.

As we have Dirichlet boundary data, we choose our sample points on the boundary. Trefethen and Gopal found for the Laplace problem that these points need to be exponentially clustered at the corners, and that around three times as many points as you have real degrees of freedom in the expansion are needed to achieve good convergence \cite{doi:10.1073/pnas.1904139116}. The selection of sample points on the boundary is a key area that we explore later on.

We quote two theorems from Trefethen and Gopal using rational functions and polynomials for the Laplace problem to demonstrate the theoretical support for these ideas \cite{doi:10.1073/pnas.1904139116}. The key goal of these theorems is to establish how exponential convergence of poles leads to root exponential error, the property we take advantage of in the Lightning Method. We also note that despite the lack of comprehensive theory for the Laplace problem, and even less theory behind the Lightning Method for the Helmholtz problem, good convergence for both have been observed in practise. Proofs for both theorems can be found in \cite{doi:10.1073/pnas.1904139116}.

We start by defining some terms. A split disk is a region given by
$$A_{\theta} = \{z \in \mathbb{C}: |z|<1, -\theta < \arg z < \theta\}.$$
$||\cdot||_{\Omega}$ denotes the supremum norm over a set, $\Omega$, and $\alpha \Omega = \{\alpha z: z \in \Omega\}$ where $\alpha \in \mathbb{C}$ is a constant.

\begin{theorem}[Convergence for a wedge]
Let f be a bounded analytic function in the slit disk $A_{\pi}$ that satisfies $f(z) = O(|z|^{\delta})$ as $z \to 0$ for some $\delta > 0$, and let $\theta \in (0, \frac{\pi}{2})$ be fixed. Then for some $\rho \in (0, 1)$ depending on $\theta$ but not $f$, there exists type $(n-1, n)$ rational functions ${r_n}$, $1\ge n < \infty$, such that
$$||f-r_n||_{\Omega} = O\left(e^{-C\sqrt{t}}\right)$$
as $n \to \infty$ for some $C>0$, where $\omega = \rho A_{\theta}$. Moreover, each $r_n$ can be taken to have simple poles only at
$$\beta_j = -e^{-\sigma j / \sqrt{n}}, \quad 0 \ge j \ge n-1,$$
where $\sigma > 0$ is arbitrary \cite{doi:10.1137/19M125947X}.
\end{theorem}

\begin{remark}
The proof is only given for $\theta < \frac{\pi}{2}$, but the result is believed to hold for $\theta < \pi$ with a more complicated choice of sample point placement for non concave corners. Trefethen and Gopal suggest placing interpolation points on the sides of the bisector, and this is what we do in our code, and what Trefethen and Gopal do in their code \cite{helmcode, doi:10.1137/19M125947X}. 
\end{remark}

\begin{theorem}[Convergence for a convex polygon ]
Let $\Omega$ be a convex polygon with corners $w_1, \dotsc, w_m$, and let $f$ be an analytic function in $\Omega$ that is analytic on the interior of each side segment and can be analytically continued to a disk near each $w_k$ with a slit along the exterior bisector there. Assume f satisfies $f(z)-f(w_k) = O\left(|z-w_k|^{\delta} \right)$ as $z \to w_k$ for each $k$ for some $\delta > 0$. There exist degree $n$ rational functions $\{r_n\}$, $1 \ge n < \infty$, such that
$$||f-r_n||_{\Omega} = O \left(e^{-C\sqrt{n}} \right)$$
as $n \to \infty$ for some $C>0$. Moreover, each $r_n$ can be taken to have finite poles only at points exponentially clustered along the exterior bisectors at the corners, with arbitrary clustering parameter $\sigma$ as in the previous theorem, as long as the number of poles near each $w_k$ grows at least in proportion to $n$ as $n \to \infty$ \cite{doi:10.1137/19M125947X}.
\end{theorem}

\begin{remark}
As with the previous theorem, this is also thought to apply for corners with interior angle greater than $\pi$, although this requires an even more careful choice of sample points. A proper handling of this requires potential theory and leads us beyond the scope of this dissertation. We will be content to just use sample points on the boundary that cluster at the corners.
\end{remark}

\section{Method Implementation Choices}

There are several choices to be made in the implementation of this method, some of which need to be chosen more carefully than others. In this section, we explore areas which need particular focus, and how choices should be made for optimal performance. During preliminary testing, we found that the most attention is needed on the Newman part and the distribution of sample points. Therefore, in this dissertation we will be keeping the length of the Runge part at $N_2=20$, and the length of the line of poles at 80\% of the distance along the interior bisector from each corner to the opposite side of the region. We have found these to be reasonable choices, but the exact choices are unimportant and mostly arbitrary. We do not explore having different numbers of poles at different corners, or changing the fixed internal point from the Runge part of the series.

Understanding these choices is crucial for good performance, and some problems will need careful attention in order to get any convergence at all. An analysis of these choices will often need to be done for each new problem considered, depending on the desired accuracy. In general, a universal set of choices will not give good results, and we will present how each problem should be analysed in the discussion. This section should not be thought of as a practical guide on method choices, but rather an exploration of the background material to such a guide.

For the purpose of simplicity, we will focus on one particular example problem: the unit square. This simple region has the corners we desire, and is the same choice made by Trefethen and Gopal in their Lightning Helmholtz code \cite{helmcode}. We will also use similar problem choices: a wave number of 20, and boundary data given by a plane wave, $\exp(-i\operatorname{Re}(20z\exp(-\frac{5\pi i}{6})))$.

We show this region in figure \ref{Square demonstration} below. Figure \ref{poles and sample points} has the poles and sample points included, marked with red dots and blue circles respectively (parameters modified for illustration purposes). In \ref{square example solution}, we have a plot of what the solution should look like, and we see the expected behaviour, such as wave scattering, diffusion, and the shadow region behind the square. Throughout this dissertation we will suppress showing the poles and sample points when showing the solution.

\begin{figure}[H]
     \centering
     \begin{subfigure}[t]{0.49\textwidth}
         \centering
         \includegraphics[width=\textwidth]{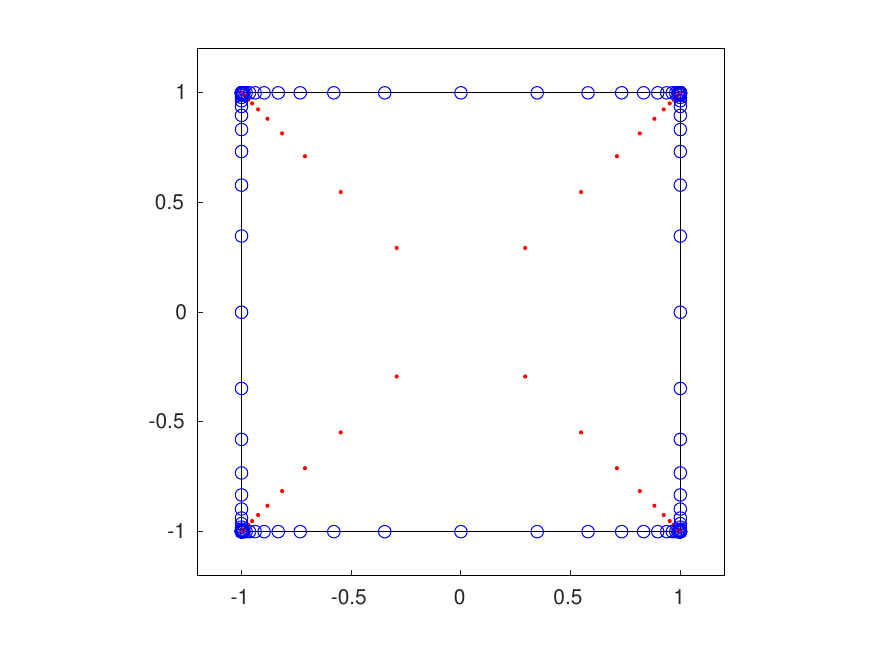}
         \caption{Poles and sample points}
         \label{poles and sample points}
     \end{subfigure}
     \begin{subfigure}[t]{0.49\textwidth}
         \centering
         \includegraphics[width=\textwidth]{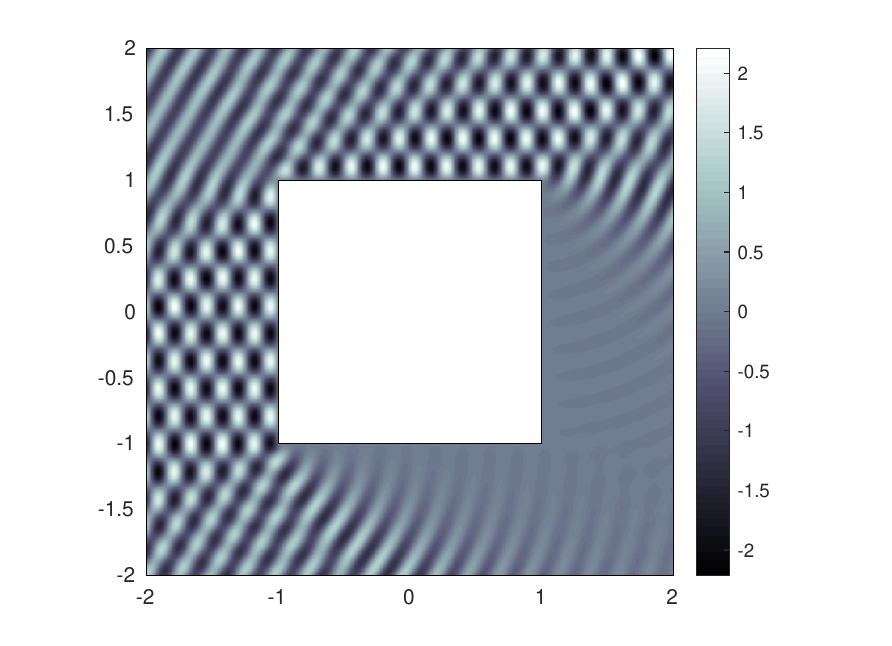}
         \caption{An example solution}
         \label{square example solution}
     \end{subfigure}
     \hfill
        \caption{A demonstration of the poles and sample points, and the solution}
        \label{Square demonstration}
\end{figure}

\subsection{Parameter and Distribution Notation}

We will denote the number of poles by $p$, and a distribution parameter by $p_r$. The distribution is given by,

$$|p_j - c_k| = \text{const}\exp(-p_r\frac{j}{\sqrt{n}})$$

Where $j = 0,1,\cdots, n-1$ indexes the poles, and $c_k$ is a corner.

The distribution of sample points is given by a strictly increasing continuous function $f(t)$ with $f(0) = 0$ and $f(1) = 1$, where 0 represents the corner, and 1 represents halfway along an edge. We use $s$ to represent the number of sample points approaching the corner from each direction (that is, $2s$ sample points per corner in total). We will be using the form $f(t) = t^A e^{4(t-1)}$, and for reference, the distribution used by Trefethen and Gopal would be $f(t) = e^{4\sqrt{n}(t-1)}$ \cite{helmcode}.

\subsection{Order of Analysis}

We first observe that the number of sample points and their rate of convergence to the corners is critical, and needs to be sufficiently high. This is to be expected, as fewer sample points means a lower resolution and less data, and as the corners are not smooth, they will need a far higher resolution. We demonstrate the issues present when these are not sufficiently high by using $A = 1$ and $s = 50$ respectively, with $p = 50$ and $p_r = 2.5$.

\begin{figure}[H]
     \centering
     \begin{subfigure}[t]{0.49\textwidth}
         \centering
         \includegraphics[width=\textwidth]{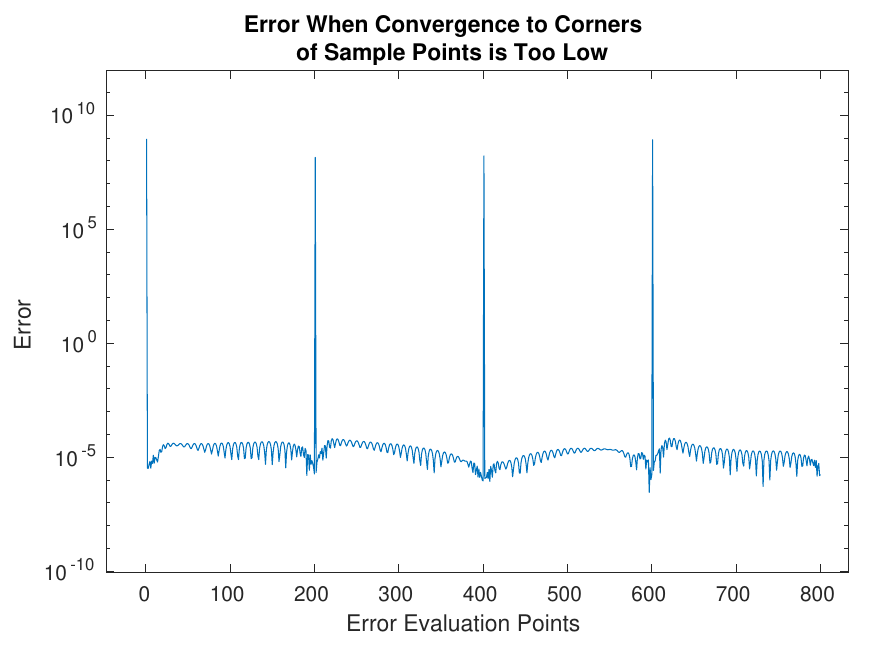}
     \end{subfigure}
     \hfill
     \begin{subfigure}[t]{0.49\textwidth}
         \centering
         \includegraphics[width=\textwidth]{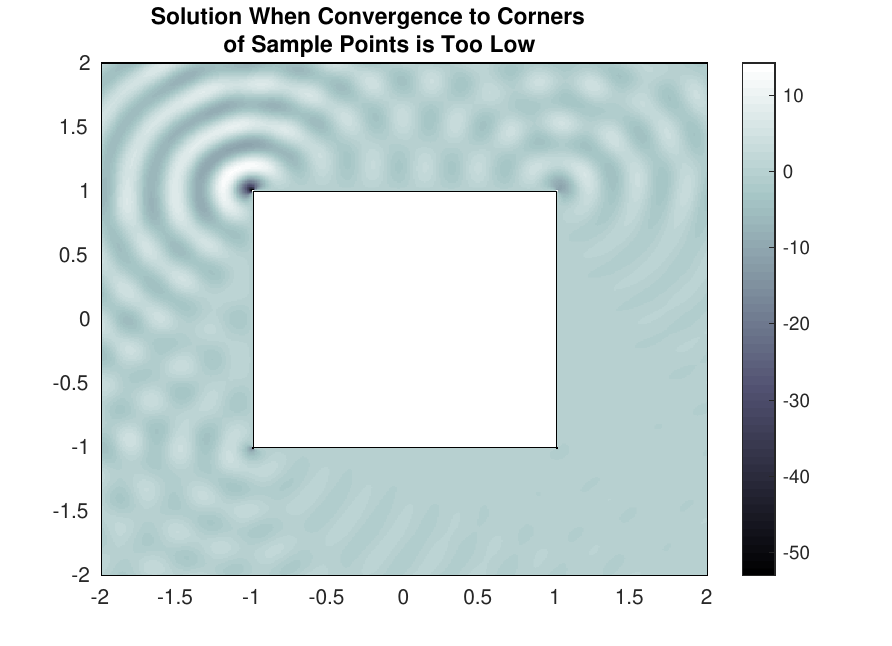}
     \end{subfigure}
     \begin{subfigure}[t]{0.49\textwidth}
         \centering
         \includegraphics[width=\textwidth]{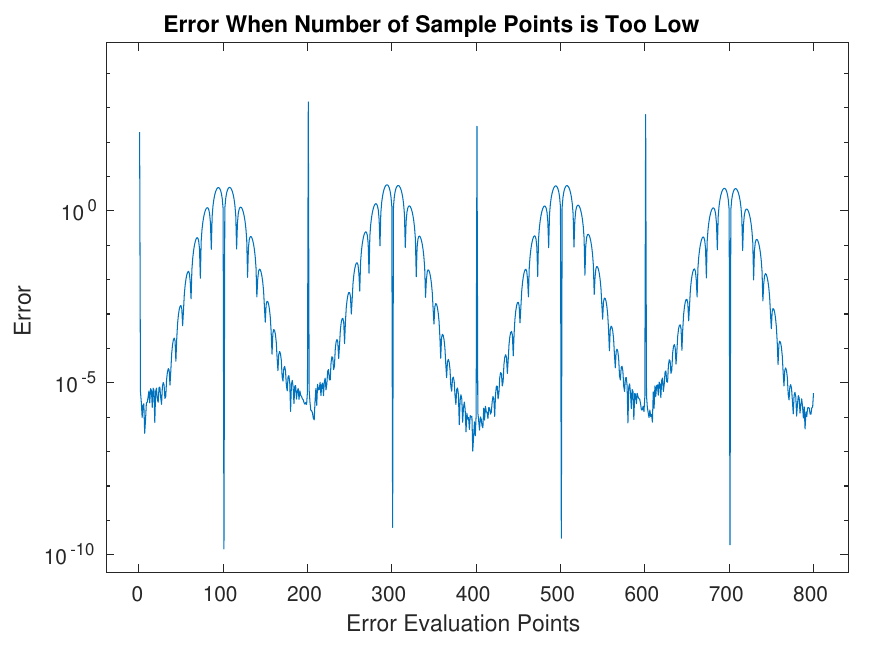}
     \end{subfigure}
     \hfill
     \begin{subfigure}[t]{0.49\textwidth}
         \centering
         \includegraphics[width=\textwidth]{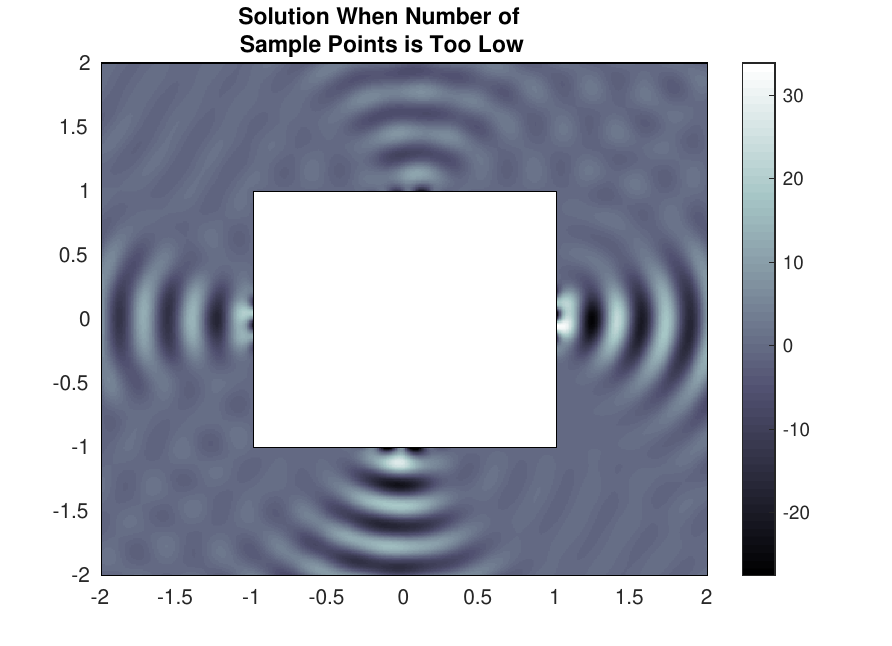}
     \end{subfigure}
        \caption{We do not get convergence to the solution when there are insufficient sample points, or they do not converge to the corners fast enough}
        \label{bad sample points}
\end{figure}

The spikes at 0, 200, 400, and 600 correspond to the corners, and the other error points are distributed evenly between them. We see that slow convergence to the corners has exceptionally bad performance, with errors up to 1 billion, and the solution is completely removed from our desired outcome. In the second situation we see that the sample points are spread too thin in the middle, and therefore we do not have any convergence here. This is manifested in the plot of the solution by fans radiating out of the midpoints, a feature we have found to be indicative of too few sample points in the centre. Despite sacrificing the middle of the edges, we still do not have a high enough resolution at the corners, highlighting how important the corners are.

Intuitively, having more sample points leads to a lower error norm, and we will see this is indeed the case (up to a point), thus we can use an excess of sample points and be confident that they are not a limiting factor in our analysis of other variables. It is also the case that the convergence rate to the corners of the sample points can be set very high without significant negative consequences.

\begin{figure}[H]
     \centering
     \includegraphics[width=\textwidth]{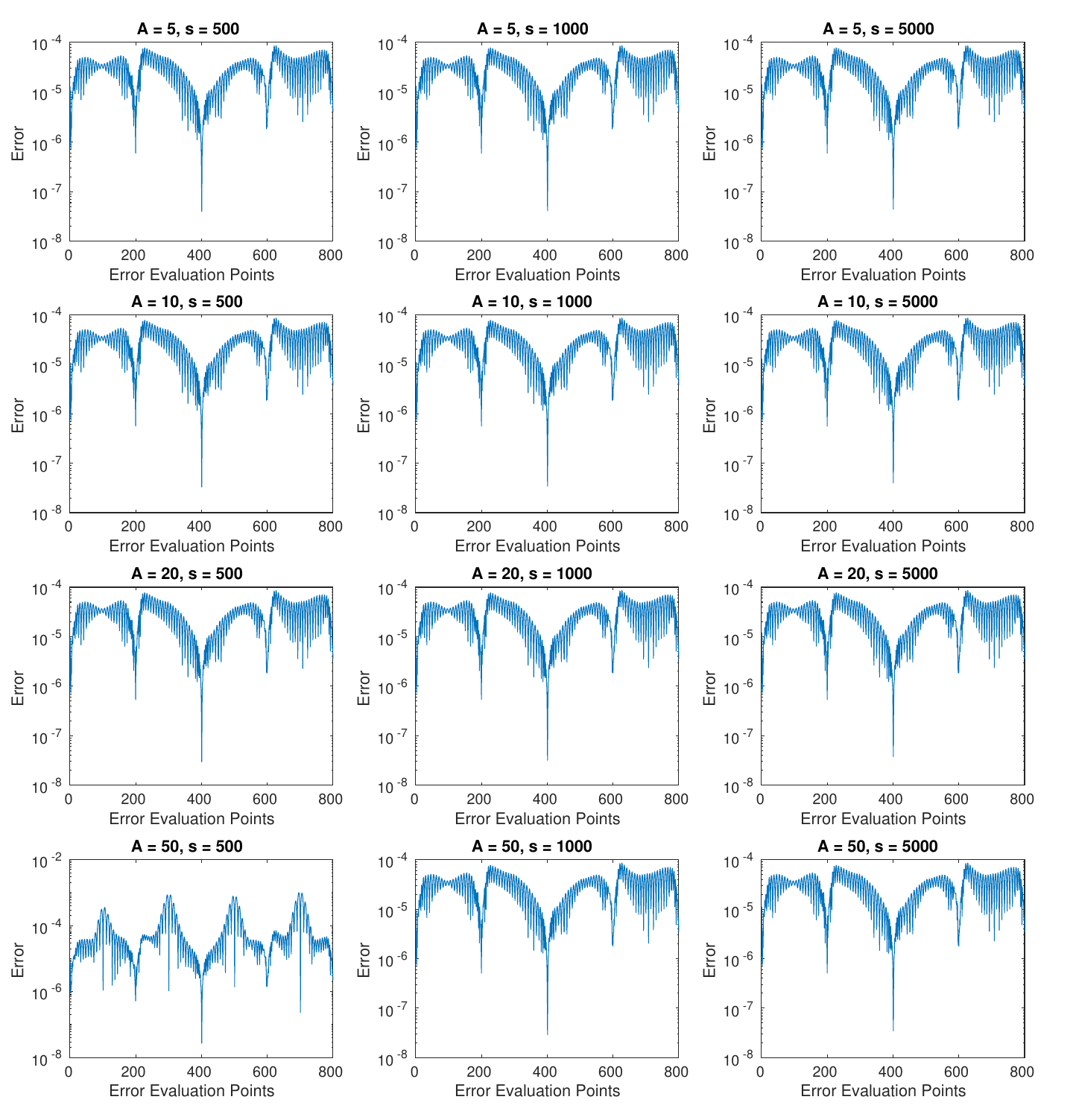}
     \caption{Choosing very high values of $A$ and $s$ still have good error profiles, with the greatest error less than $10^{-2}$}
\end{figure}

The above figure demonstrates that the method is well behaved even when using large parameter choices. We note the $A = 50$ and $s = 500$ plot is starting to suffer from a lack of sample points, which shows these two parameters need to be compatible with each other. We can use this fact to our advantage, as it allows us to eliminate issues due to sample points from our analysis of other variables. Thus, we can consider the poles and sample points in sequence rather than together, greatly simplifying and speeding up our computations. We ensure that the convergence is sufficiently fast not to interfere with any other experiments by choosing $A = 4$ and $s = 500$.

\subsection{Distribution of Poles}

We need to be more careful when considering the distribution of poles than we were with our initial analysis of the distribution of sample points. As we are introducing poles very close to the corners, we suspect that this could lead to problems due to the singular nature of the poles and the finite precision of the computations. We confirm this by increasing the number of poles until we have problems, and plotting the resulting error profiles.

\begin{figure}[H]
     \centering
     \includegraphics[width=\textwidth]{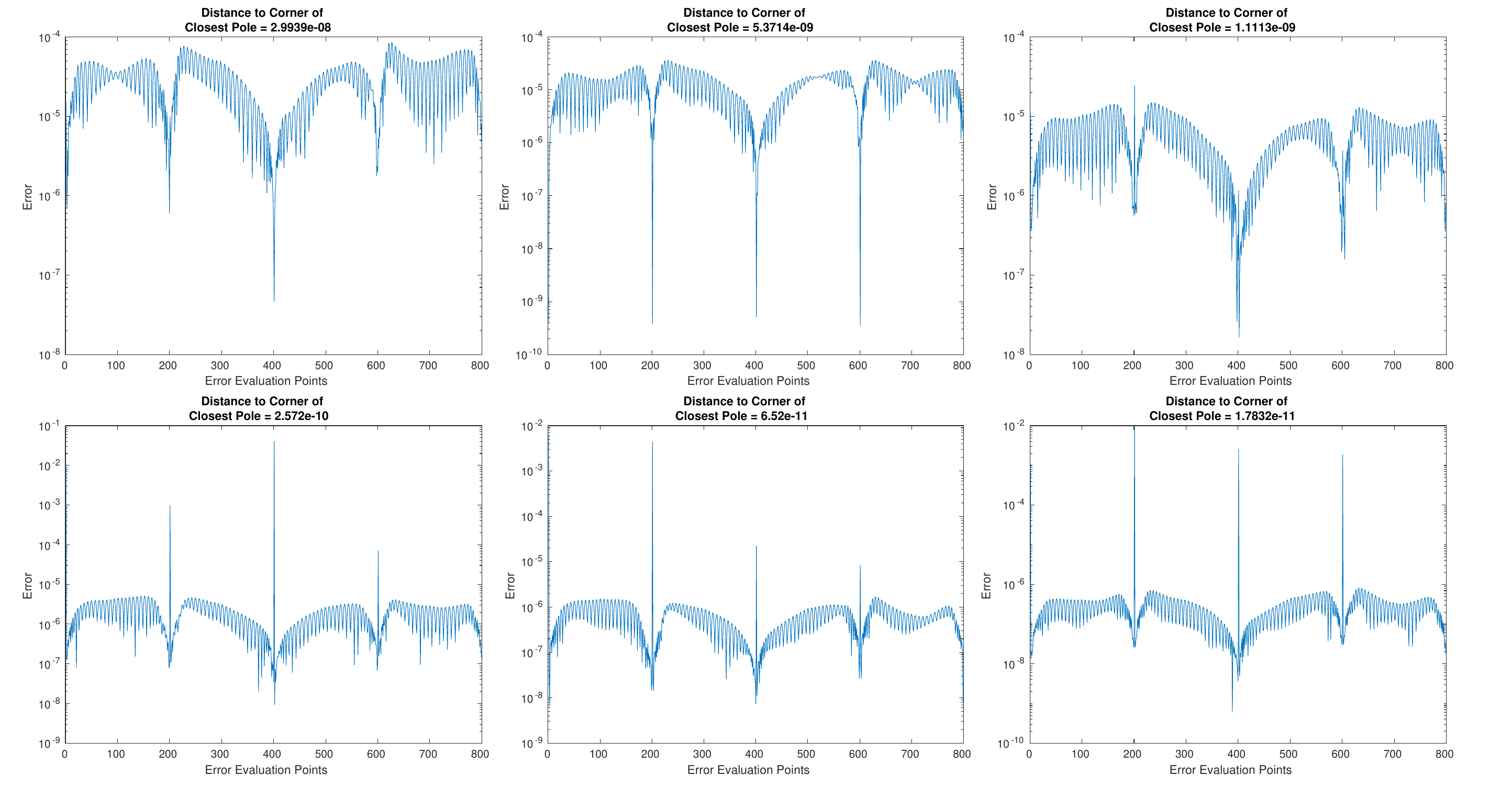}
     \caption{Relationship between distance to the corner and shape of the error profile}
\end{figure}

We can clearly see that when the poles get too close we have very poor behaviour at the corners compared to the rest of the error evaluation points. There appears to be a turning point in the behaviour around $10^{-9}$, which we confirm in figure \ref{minimum distance vs error}.

\begin{figure}[H]
     \centering
     \includegraphics[width=\textwidth]{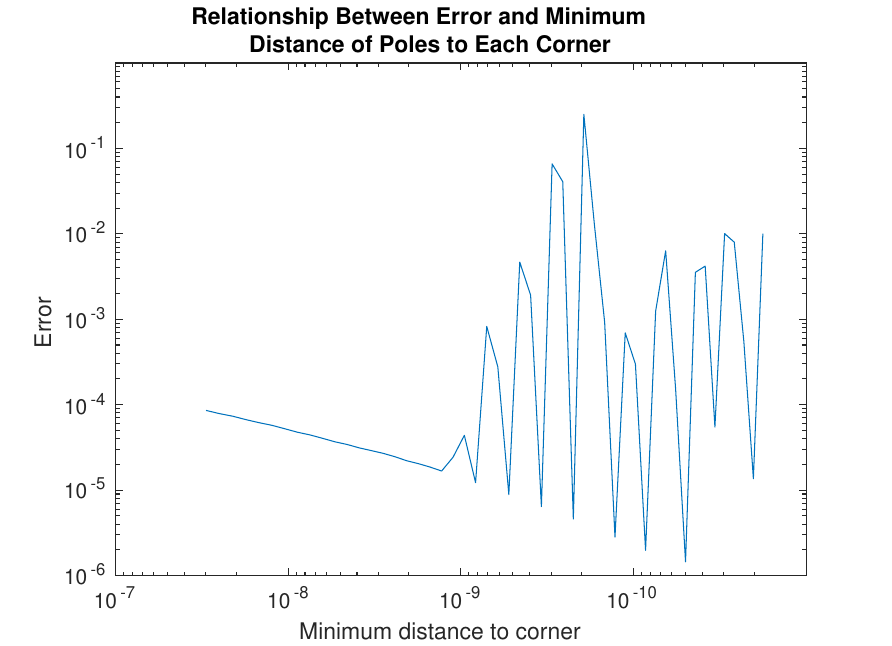}
     \caption{Relationship between distance to the corner and shape of the error profile}
     \label{minimum distance vs error}
\end{figure}

We attain the rather surprising result that after $10^{-9}$, there are semi-regular breaks in the poor behaviour where the error follows the trajectory it was on beforehand. We will not explore this further in this dissertation, although if one could understand the precise cause of the brief returns to low errors, this could perhaps be used to gain another digit of accuracy. We echo the comments of Trefethen and Gopal that improvement can be found in better choices of pole placement \cite{doi:10.1073/pnas.1904139116}. We have a limiter in our code that removes any poles closer than $10^{-9}$ to their corners which means adding additional poles past this point will not lower the error.

Now that we are safe from being too close to the corner we explore the rate of convergence to the corners of the poles. We expect that very slow or very fast convergence will underperform, meaning there should be a middle ground. We plot their error profiles in figure \ref{pole convergence rate error profiles} to see what can go wrong if the rate is chosen poorly.

\begin{figure}[H]
     \centering
     \includegraphics[width=\textwidth]{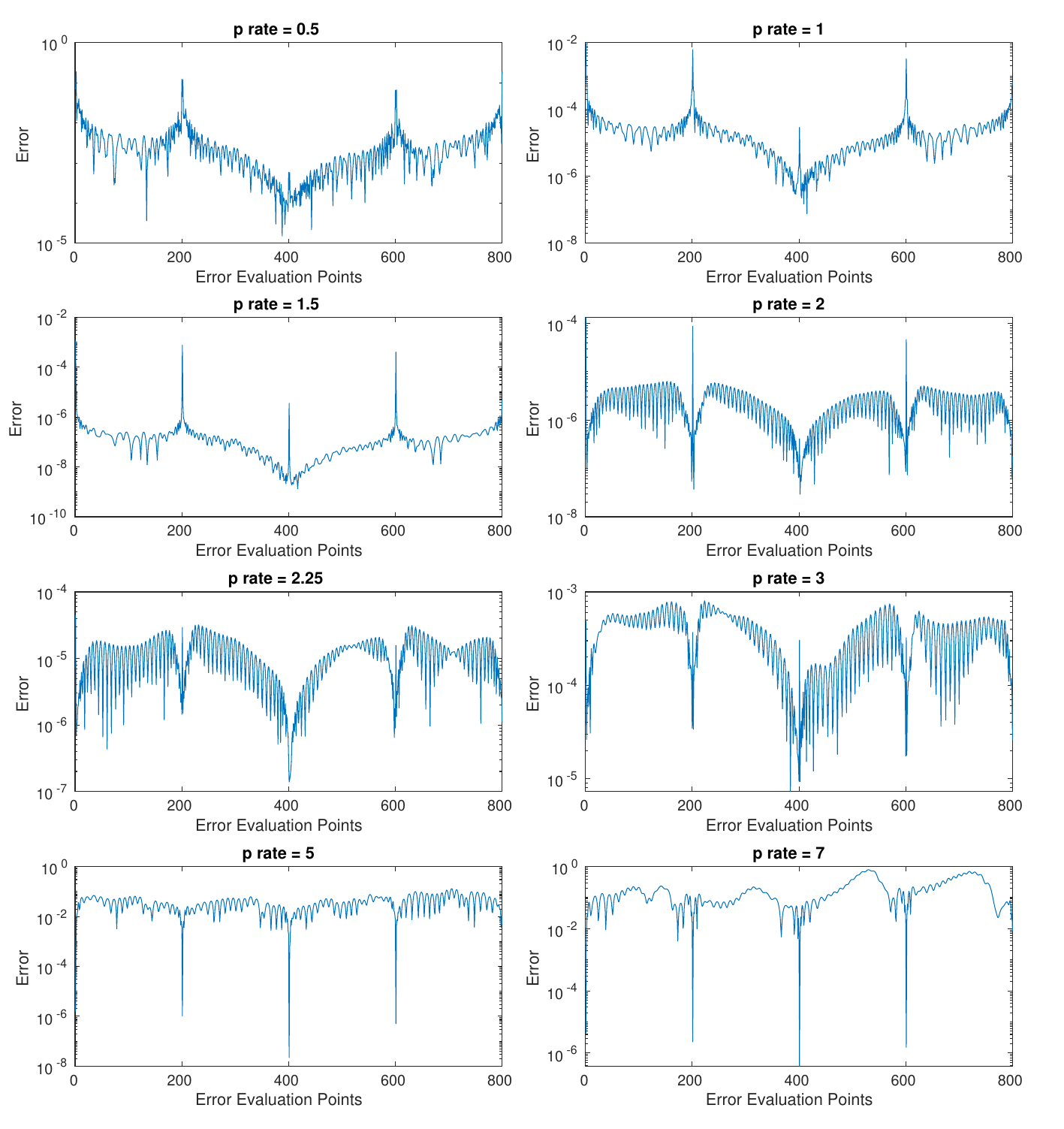}
     \caption{Relationship between pole convergence rate and shape of the error profile}
     \label{pole convergence rate error profiles}
\end{figure}

For low rates, the corners have not been given enough attention, and they have large spikes in error. While the error on the rest of the boundary is far better than the corners, it is still considerably lower in the $p_r = 0.5$ plot than the others. Looking at the higher rates, it is clear that more poles are needed near the centre of the region as the error is poor everywhere but the corners. Interestingly, the error at the corner is no smaller for large $p_r$ than when $p_r = 2.25$. From these plots we deduce that a uniform error on the boundary leads to the lowest error in the infinity norm. We perform a higher resolution analysis of the error in the infinity norm to see the relationship with pole convergence rate more clearly, and show our results in figure \ref{best choice of convergence rate}.

\begin{figure}[H]
     \centering
     \includegraphics[width=0.9\textwidth]{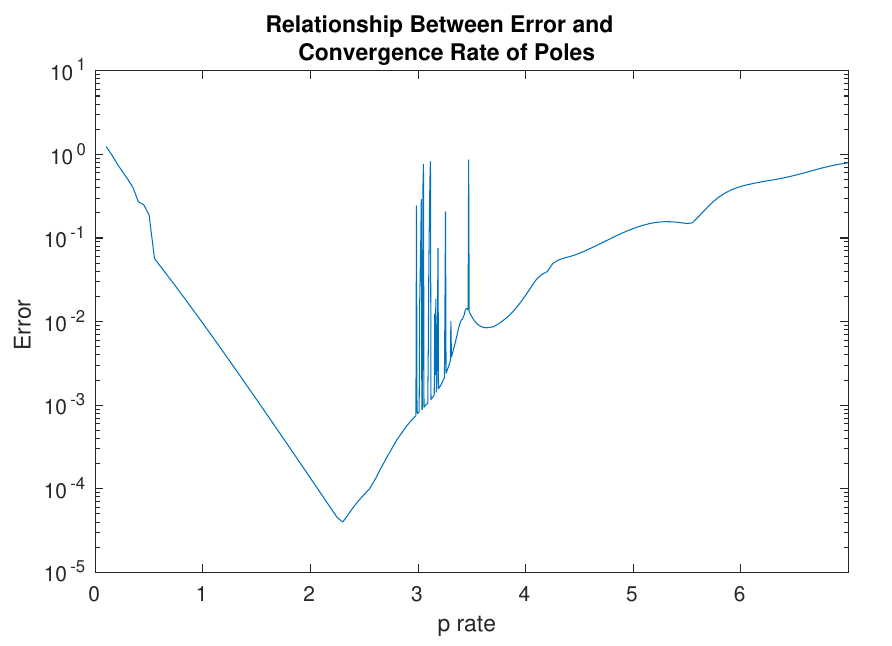}
     \caption{There is a clear choice of best convergence rate to use}
     \label{best choice of convergence rate}
\end{figure}

\begin{figure}[H]
     \centering
     \includegraphics[width=0.9\textwidth]{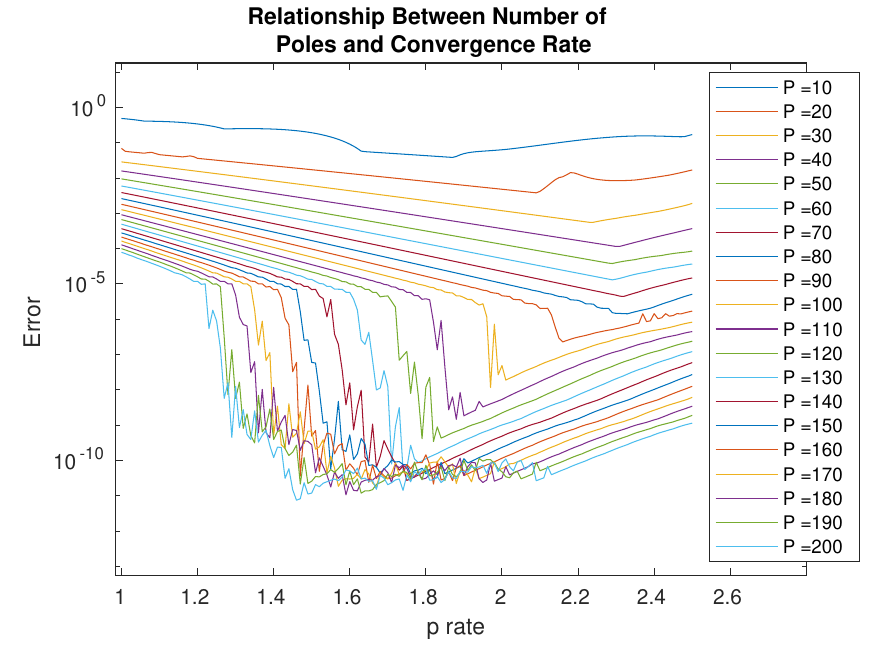}
     \caption{The relationship between error, pole convergence rate, and number of poles}
     \label{poles vs pole rate}
\end{figure}

\begin{figure}[H]
     \centering
     \includegraphics[width=0.9\textwidth]{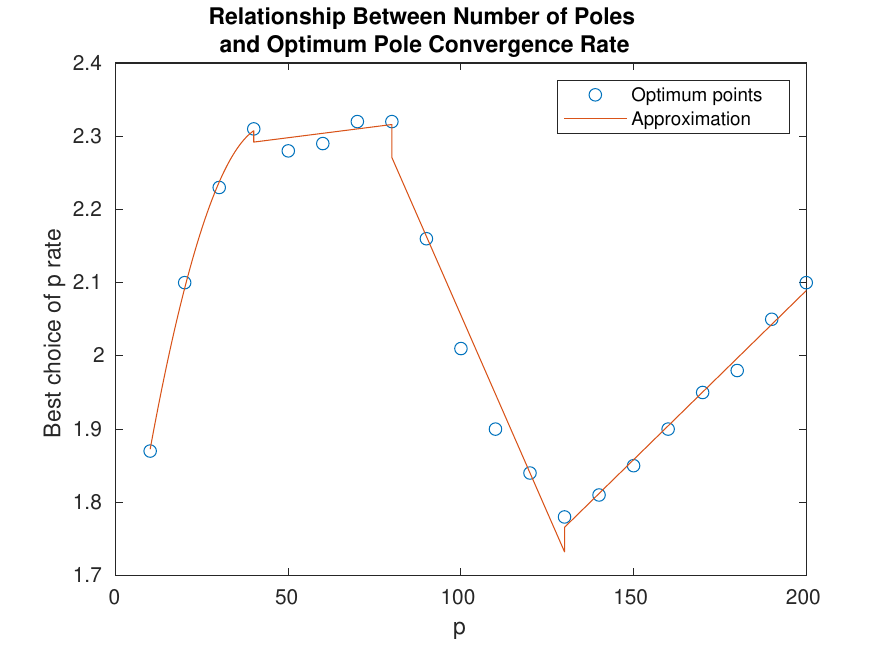}
     \caption{The optimal pole convergence rates found from the previous plot}
     \label{pole vs rate vs error}
\end{figure}

In figure \ref{pole vs rate vs error} we see that for $p_r \le 1.2$ and $p_r \ge 2.3$ all the lines exhibit similar behaviour, and as expected, the higher the pole number, the smaller the error. Interestingly we see a split in behaviour at the 80 pole per corner mark, where the lines for $p > 80$ have a sudden drop before following the upward trend of the $p \le 80$ family. As $p$ increases, this drop point recedes further backwards, and at $p = 130$ we hit the limit imposed on us by the finite precision of floating point arithmetic. To give a small safety margin, we choose a point slightly away from the jagged area after the drop as our minimum for each value of $p$. We plot the minima against the number of poles per corner, and see four clear sections. We use a least squares regression to find a piecewise equation that fits the data, which we label as `approximation' in the plot.
\begin{equation}
f(p) = 
    \begin{cases}
    \begin{alignedat}{3}
    &-0.000375p^2+0.0333p+1.578, & &\quad \text{if} & p &< 40\\
    &0.0006p+2.268, & &\quad \text{if} & 40 \le p & < 80 \\
    &-0.0108p+3.133, & &\quad \text{if} & 80 \le p &< 130\\
    &0.00462p+1.165, & &\quad \text{if} \quad & 130 \le p &
    \end{alignedat}
    \end{cases}
\end{equation}

We also see that there is no real advantage to increasing the number of poles past 130, as we have already hit the limit. There is a small region around $p_r = 1.5$ where the $p = 200$ line dips slightly below all the other lines, but apart from fluctuations like this, the best that this method can achieve is 10 digits of accuracy. This is in line with Trefethen and Gopal's comments in \cite{doi:10.1073/pnas.1904139116}.

\subsection{Distribution of Sample Points}

Now that we have a good understanding of the poles, we can move on to the sample points. As we have already seen, we can easily find choices for the distribution and number of sample points that lead to good results, but every sample point adds a row to the problem matrix. This means that any unnecessary sample points will slow down the computation, potentially to a significant degree in the case of extreme choices. We also note that while having many sample points does not appear to hinder convergence, it is possible that choosing fewer points can be better, as we will see. We plot the number of sample points against the error for a range of values of $p$ and $A$.

\begin{figure}[H]
     \centering
     \includegraphics[width=0.9\textwidth]{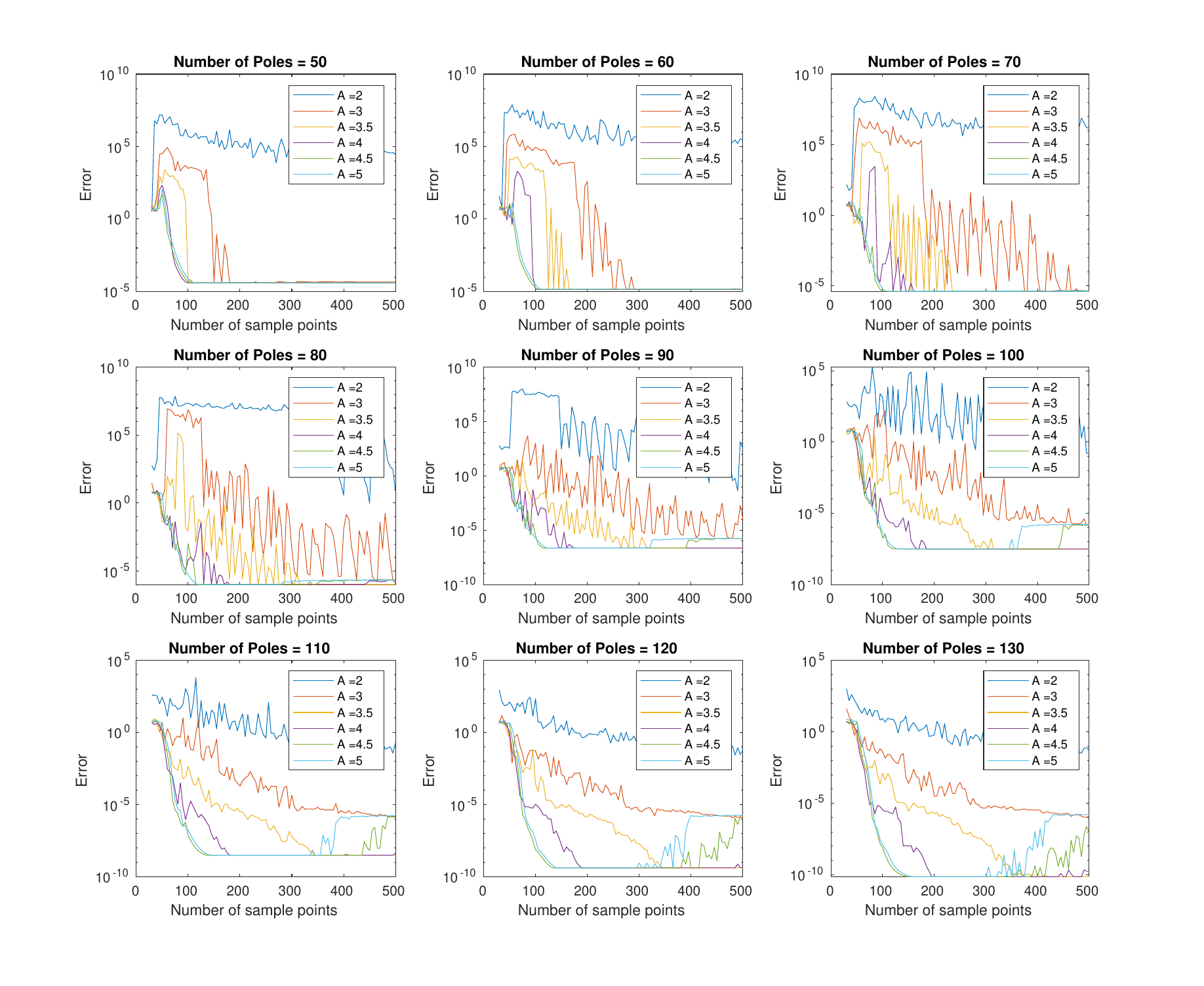}
     \caption{The relationship between $p$, $A$, $s$, and the error}
\end{figure}

We notice that convergence is facilitated when there are fewer poles, which is mainly shown by looking at the behaviour of the $A = 3$ line or where the $A = 4$ line intersects the axis. In all the plots the $A = 4.5$ and $A = 5$ lines converge very fast to the minimum error, but curiously, after a period of remaining level, they start climbing again to around $10^{-5}$ for the larger values of $p$. After reaching this new error, they level out again. We have also observed this for higher values of $A$, and note that the larger $A$ is, the earlier the line starts to deviate from the minimum error. We even see the $A = 4$ line show signs of this behaviour in the last two plots, and it may be the case that this happens to all values of $A$ if the number of poles per corner is high enough. Luckily this does not affect us as we do not plan on using extremely large numbers of poles, and thus we will not explore this area in this dissertation.

Using $A = 5$ and $s = 160$ guarantees good performance, at least for this example problem. Increasing the number of sample points is typically safe, but should be done with $A = 4$ to make sure the error stays at a minimum. $A = 4$ and $s = 200$ is also a good choice to use, although if low numbers of poles are used it may be worth lowering the number of sample points to decrease computation time. When generating plots, the solution does not need to be more accurate than two or three digits, and the computation time for generating the plot will be far greater than the time taken to find the solution, so only a small percentage of time saving could be achieved by lowering the number of sample points. We advise that it is not worth fine-tuning the number of sample points much beyond our recommendations here, unless the solution is being evaluated on a small number of points.

We are now in a position to discuss the claim of Trefethen and Gopal that the sample points need to be clustered exponentially near the corners. We use functions on $[0, 1]$ to describe distributions, so we will use the map $F_n:[0, 1] \to [0, n]$ given by $F_n(t) = 1-\frac{t}{n}$ to convert them to the same territory as Trefethen and Gopal's.

That is, $S = \left\{ f\left(\frac{j}{n}\right) : j = 0, \frac{1}{n},\cdots, 1\right\} = \left\{ f(F(k)) : k = 0,1,\cdots, n \right\}$ describes the same set of sample points. We show Trefethen and Gopal's distribution and ours below.
\begin{align*}
    \text{Trefethen and Gopal's:} \quad & \exp \left(-\frac{4t}{\sqrt{n}}\right)\\
    \text{Ginn's:} \quad & \left( 1-\frac{t}{n} \right)^A \exp \left(-\frac{4t}{n} \right)\\
\end{align*}
\begin{figure}[H]
     \centering
     \includegraphics[width=\textwidth]{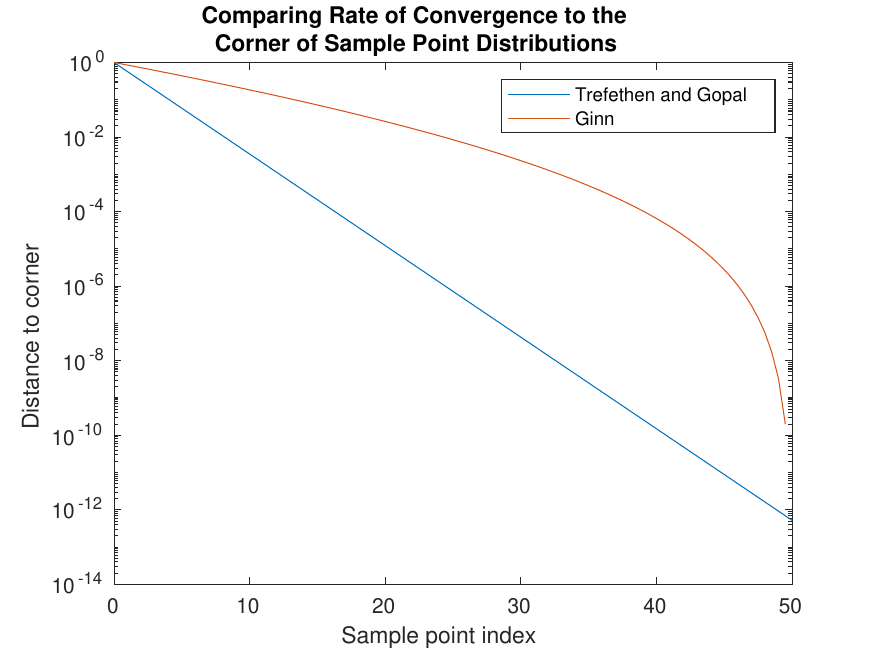}
     \caption{The two distributions of sample points being considered}
     \label{distributions of sample points}
\end{figure}

Taking logs we have,
\begin{align*}
    \text{Trefethen and Gopal's:} \quad & -\frac{4t}{\sqrt{n}}\\
    \text{Ginn's:} \quad & A \log \left( 1-\frac{t}{n} \right) -\frac{4t}{n}\\
\end{align*}
We see that the slope of both lines in figure \ref{distributions of sample points} is non increasing, therefore both exhibit the exponential convergence prescribed by Trefethen and Gopal \cite{doi:10.1073/pnas.1904139116}, although Trefethen and Gopal's converges at a faster rate, and this becomes even more pronounced as $n$ increases. Trefethen and Gopal also use a set of equally spaced points superimposed with the exponentially distributed points, presumably to get higher resolution in the middle of the edges, although we do not require this in ours. When these points are removed, we see the familiar fan shapes expanding out from the midpoints of the edges. We encountered this phenomenon in our earlier experiments with low sample point density around the middle of the edges.Now that we have chosen our parameters optimally, we see the expected root exponential convergence.

\begin{figure}[H]
    \centering
    \includegraphics[width=0.8\textwidth]{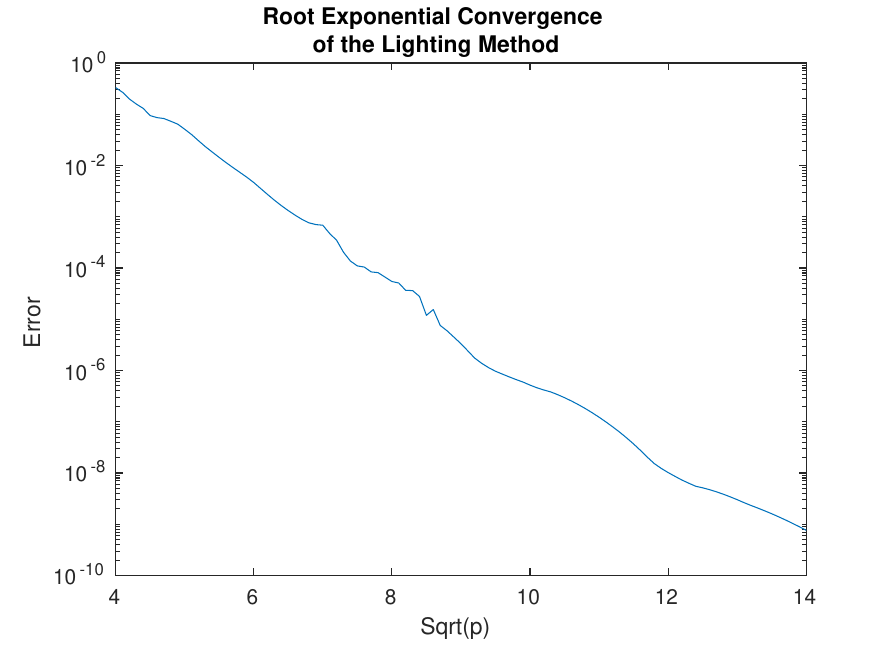}
    \caption{We have a straight line when plotting $\sqrt{p}$ against $\log(\text{error})$, which demonstrates the celebrated convergence rate of the method}
\end{figure}

\subsection{Alternative Solution Forms}\label{alternative solution forms}

So far we have used the same form of solution as Trefethen and Gopal \cite{helmcode}, given by equation \eqref{Form of solution}, although there are other expansions we could have used. Firstly, we recall the expansion inspired by \cite{multipole} that includes terms of the form $H_n^{(1)}(k|z|) \frac{z_*^{-n}}{|z_*^{-n}|}$ that we looked at earlier \ref{Form of solution extended runge}. However, this increases computation time significantly while doing very little to reduce the error. It performs better on regions where the Runge part of the solution needs more focus, although in these cases we have found it is better to use the original form with a few more terms in the Runge part instead. It was for these reasons that we used equation \eqref{Form of solution} instead.

One area we have found improvement in is forms that use higher order Hankel functions in the Newman part. We will call these other forms m-Newman forms, and note that Trefethen and Gopal's form is the case $m = 1$. The motivation behind this is that for a problem dominated by the Newman part, such as our test problem, we want more terms in the series dedicated to resolution of the corner singularities. We give a general m-Newman form below:

\begin{equation}
    \sum_{j=1}^{N_1} \sum_{n=0}^{m} \left(a_j^n H_n^{(1)}(k|z_j|) \frac{z_j^n}{|z_j^n|} \right) + \sum_{n=0}^{N_2}c_nH_n(k|z_*|) \frac{z_*^n}{|z_*^n|}
    \label{m newman form}.
\end{equation}

The main strength of these forms is on problems where it is hard to reach any convergence at all, and the error stays stubbornly around $O(1)$. For such problems, it is tempting to use a large number of poles that get very close to the corner, although we saw earlier that this leads to problems and we need to insert a limit on how close the poles can get to the corners. Using higher order m-Newman forms allows us to get more out of each pole in exchange for longer computation times, and thus achieving better results. We have also found that as m increases, the optimal rate at which the poles converge to the corners decreases, which allows to use even more poles. This effect is not very large, but it can be just enough to be able to get one or two digits of accuracy which is sufficient to produce a viable plot. We will be using this in the next section when necessary, although we will defer a thorough analysis to further research.

\section{Verifying and Exploring Wave Behaviour}

In this section we will explore whether our solution waves behave as we expect them to. This will entail looking at various wave behaviours for a variety of regions, and verifying whether our produced solutions match. We will omit details on choices of parameters and the maximum error, but guarantee that all solutions are sufficiently accurate for plotting purposes.

\subsection{Superposition}

Perhaps the most basic property of waves is that of superposition. This is where the amplitude of multiple waves occupying the same point in space add together to create a resultant wave. This is built in to our solution by the linearity of the Helmholtz equation, where if $u_1(z)$ and $u_2(z)$ are solutions, then $u(z) = u_1(z)+u_2(z)$ will also be a solution.
\vspace{-2mm}
\begin{figure}[H]
     \centering
     \begin{subfigure}[t]{0.45\textwidth}
     \label{superposition a}
         \centering
         \includegraphics[width=\textwidth]{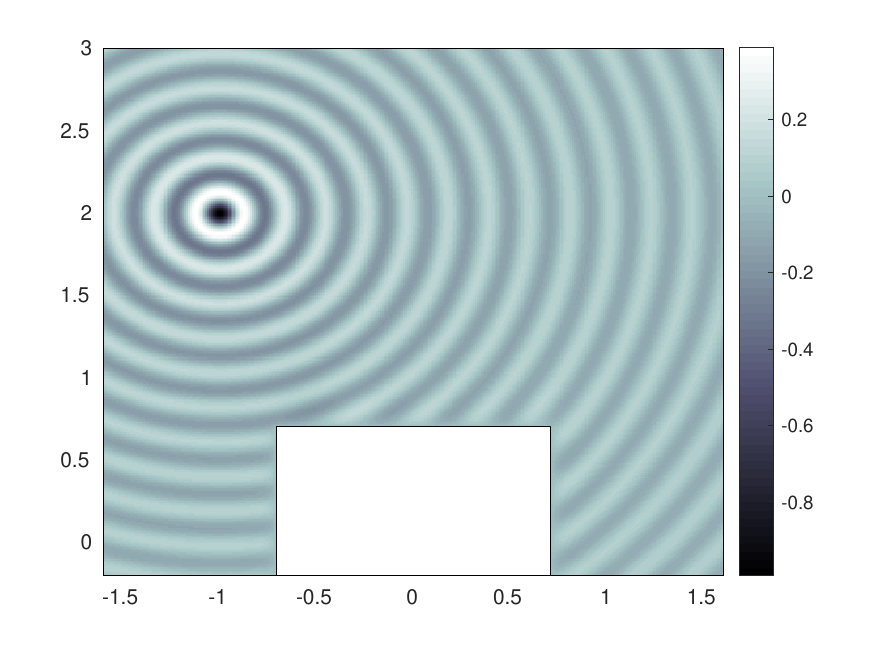}
         \vspace{-9mm}
         \caption{A point source at (-1, 2)}
     \end{subfigure}
     \hfill
     \begin{subfigure}[t]{0.45\textwidth}
     \label{superposition b}
         \centering
         \includegraphics[width=\textwidth]{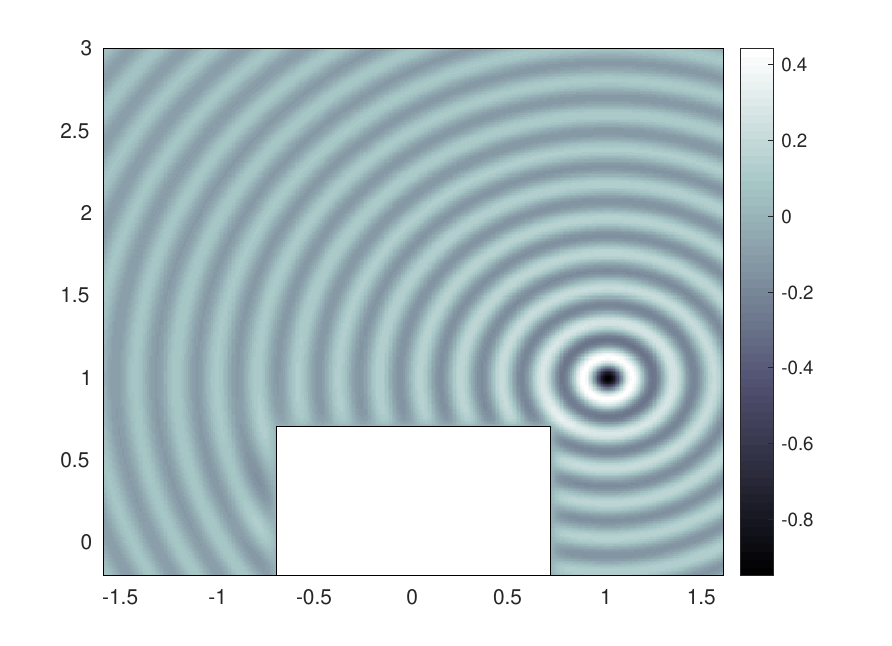}
         \vspace{-9mm}
         \caption{A point source at (1, 1)}
     \end{subfigure}
     \begin{subfigure}[t]{0.45\textwidth}
     \label{superposition c}
         \centering
         \includegraphics[width=\textwidth]{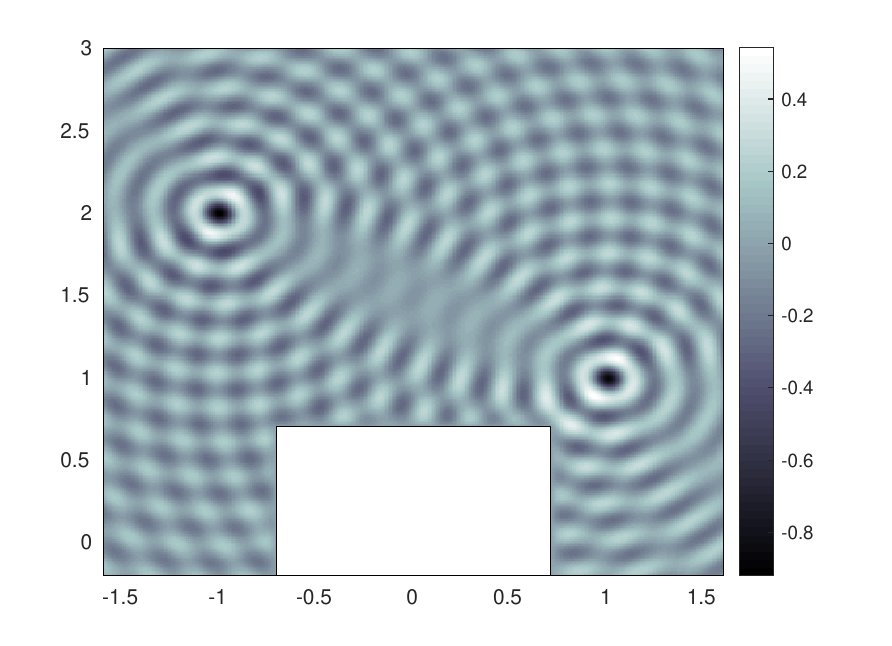}
         \vspace{-9mm}
         \caption{Both point sources}
     \end{subfigure}
     \hfill
     \begin{subfigure}[t]{0.45\textwidth}
     \label{superposition d}
         \centering
         \includegraphics[width=\textwidth]{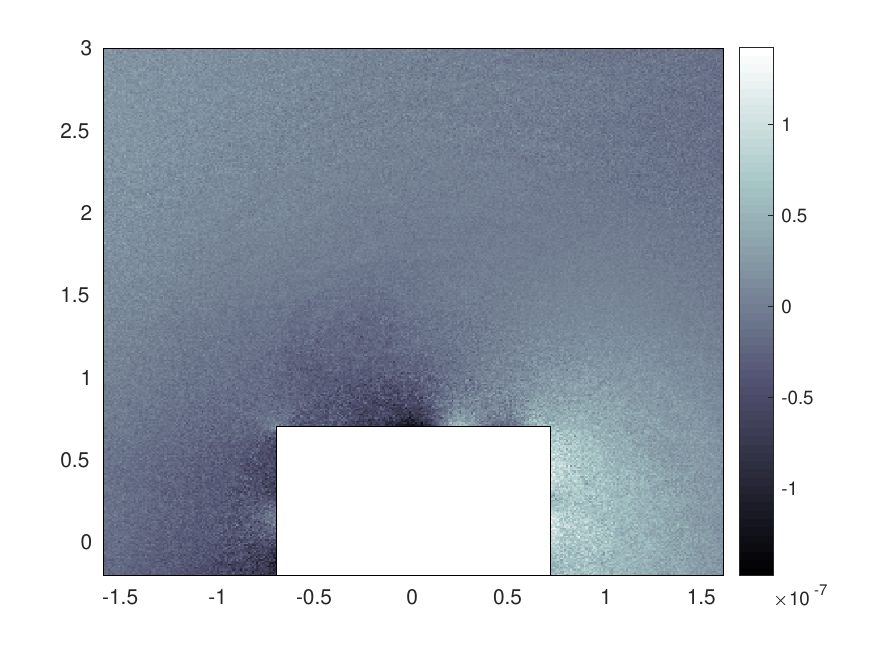}
         \vspace{-9mm}
         \caption{$u_3(z)-(u_1(z)+u_2(z))$}
     \end{subfigure}
        \caption{Demonstrating the difference between a linear combination and the true solution of the combination is approximately 0}
\end{figure}

We verify that this is indeed the case by solving over a region with three sets of boundary data: one where $u_1(z) = f_1(z)$, another with $u_2(z) = f_2(z)$, and finally $u_3(z) = f_1(z)+f_2(z)$ where $z \in \delta \Omega$. We then plot the the reflected part of $u_3(z)-(u_1(z)+u_2(z))$, and we should get approximately 0 from the two solutions cancelling out.

We observe that the error over the region is on the scale of $O(10^{-7})$, verifying our hypothesis that the result of the Lightning Method obeys the law of superposition as this is very close to 0. Equivalently, the Lightning Method operator is approximately linear\footnote{This statement is true for problems defined on the same region, and only up to some error depending on the problem parameters and region. We must also have convergence for this statement to hold as well.}.

\subsection{Diffraction}

When a wave enters a wider region, such as coming out of a slit or a tunnel opening out, the wave spreads out. We will look at this in the case of a wave propagating around a corner. It is a well known fact that the degree of diffraction decreases with wave number, and the wave number remains constant. For illustration purposes we have increased the sensitivity of the colour scale to be able to see the fainter parts of the wave. Typically this is demonstrated with an infinite wall, although so that we have convergence we will settle for an approximation of this with a wall of height 3. Unfortunately this means we will see some waves coming around the other side of the wall.

\begin{figure}[H]
     \centering
     \begin{subfigure}[t]{0.45\textwidth}
         \centering
         \includegraphics[width=\textwidth]{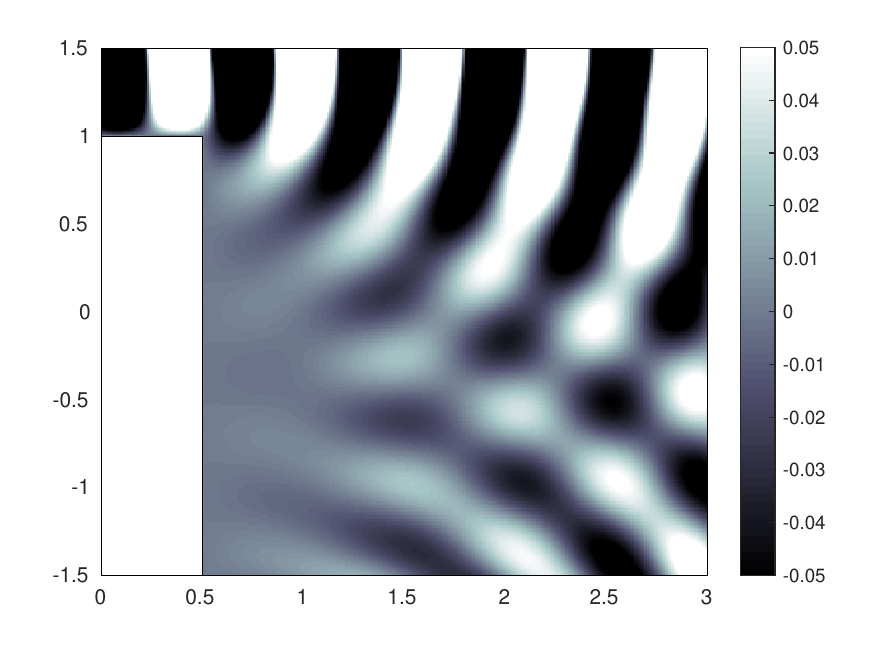}
         \caption{Wave number = 10}
     \end{subfigure}
     \hfill
     \begin{subfigure}[t]{0.45\textwidth}
         \centering
         \includegraphics[width=\textwidth]{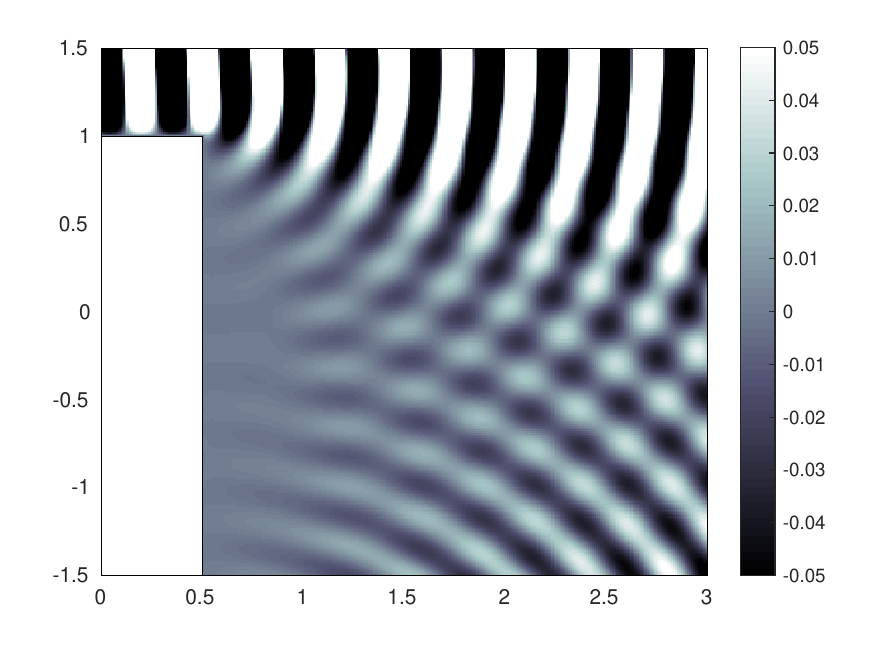}
         \caption{Wave number = 20}
     \end{subfigure}
     \begin{subfigure}[t]{0.45\textwidth}
         \centering
         \includegraphics[width=\textwidth]{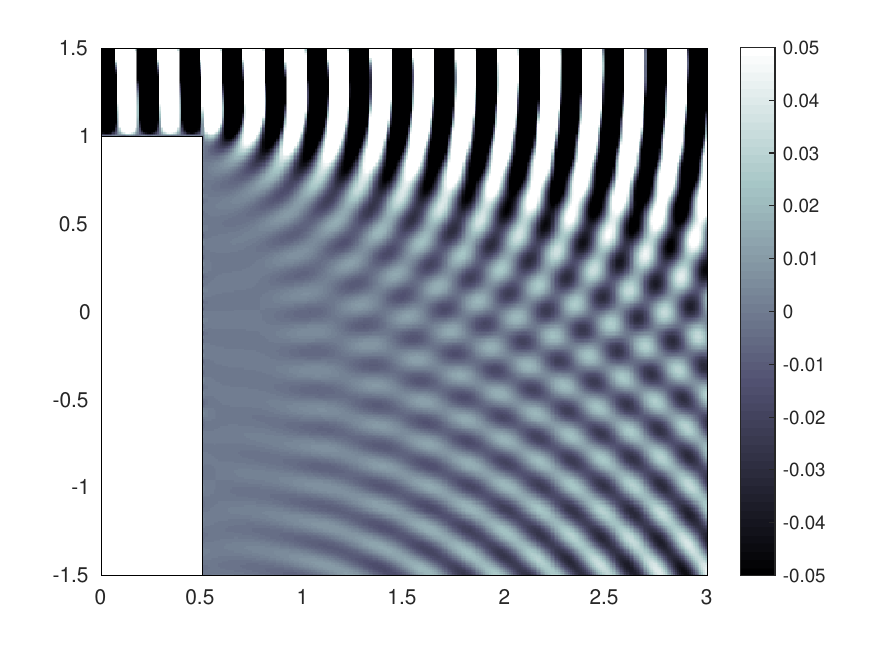}
         \caption{Wave number = 30}
     \end{subfigure}
     \hfill
     \begin{subfigure}[t]{0.45\textwidth}
         \centering
         \includegraphics[width=0.8\textwidth]{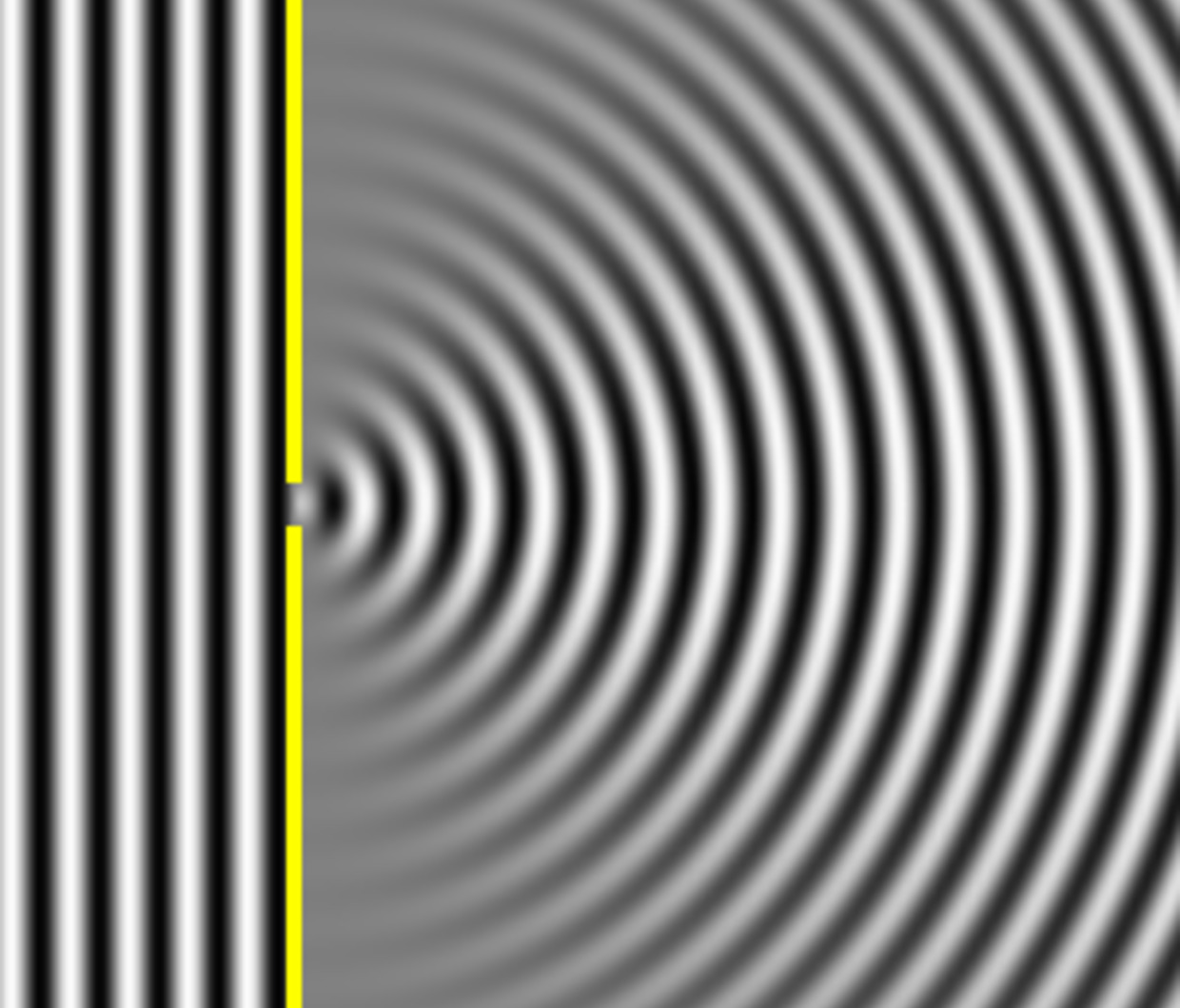}
         \caption{A wave passing through a slit \cite{Diffraction}}
         \label{example diffraction}
     \end{subfigure}
        \caption{Waves of varying wave number diffraction around a corner, and an accurate example of diffraction}
\end{figure}

\newpage

We see in the plots above that all the waves diffract around the corner, filling all the space around the corner, as predicted by the Huygens–Fresnel principle. There is also a difference in the size of the angle to which the most prominent part of the wave diffracts, and this also matches our prediction. The wavelength between the circular arcs is also observed to be constant, which follows the expected behaviour mentioned above. Decay of the amplitude as the angle of diffraction increases to $\frac{\pi}{2}$ is also evident, and our plots match the diffraction behaviour of the example plot in \ref{example diffraction} well.

\subsection{Reflection}

The next avenue of investigation is whether or not the solution wave reflects off surfaces in the manner that we expect. Here we aim to verify that the law of reflection holds, that is, angle of incidence = angle of reflection. Below we demonstrate the components of the solution by plotting the incident and reflected waves both separately and combined. We observe that the wave length of the reflected part is unchanged, and this decomposition will also aid us in determining the angles more accurately.

\begin{figure}[H]
     \centering
     \begin{subfigure}[t]{0.32\textwidth}
         \centering
         \includegraphics[width=\textwidth]{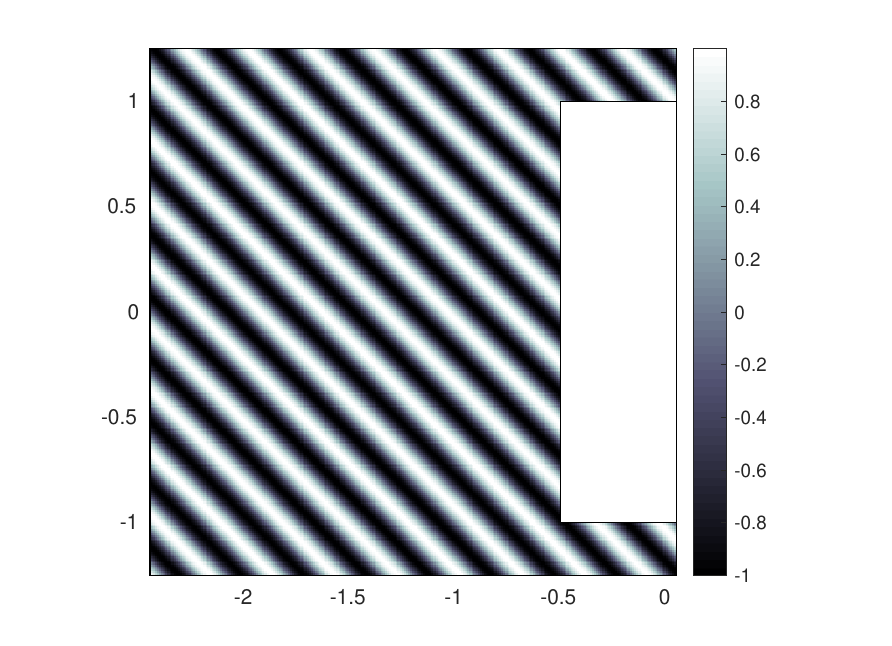}
         \caption{Incident wave}
     \end{subfigure}
     \begin{subfigure}[t]{0.32\textwidth}
         \centering
         \includegraphics[width=\textwidth]{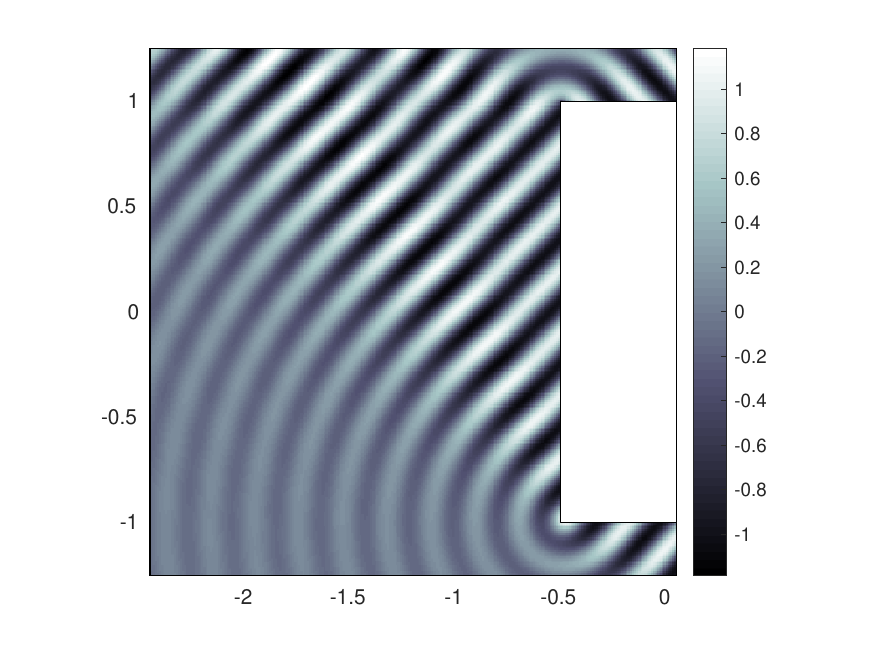}
         \caption{Reflected wave}
     \end{subfigure}
     \begin{subfigure}[t]{0.32\textwidth}
         \centering
         \includegraphics[width=\textwidth]{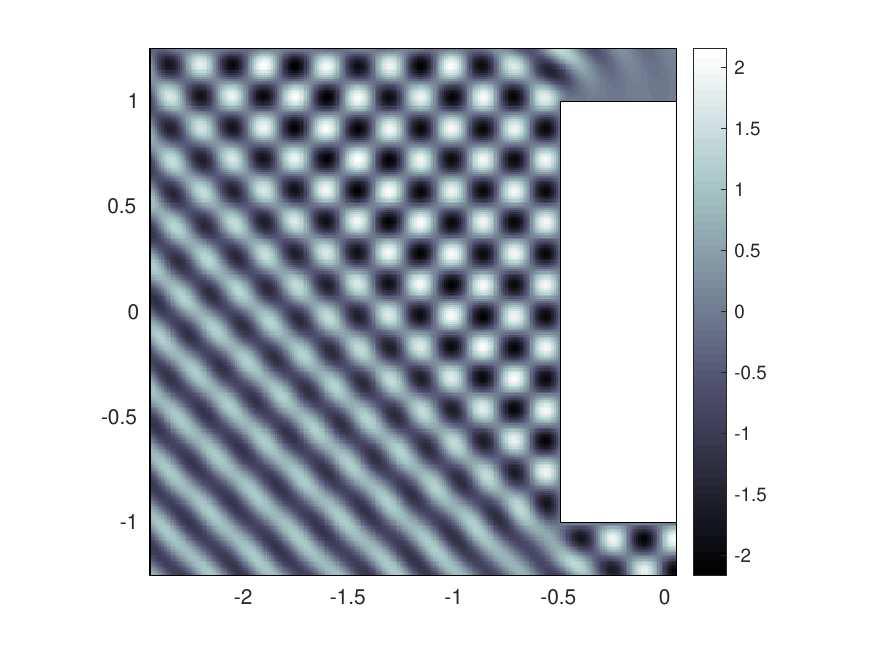}
         \caption{Combined solution}
     \end{subfigure}
        \caption{An example showing the components of the wave hitting a surface at an angle of $\frac{\pi}{4}$}
\end{figure}

Below we record our results in a table, and then plot them on a graph. We note that the angles were measured by hand in GIMP \cite{gimp}, and thus have some error.

\begin{center}
\begin{tabular}{|c|c|}
    \hline
    Angle of Incidence (rad) & Angle of Reflection (rad) \\
    \hline
    0 & 0 \\
    $0.05\pi$ & $0.049\pi$ \\
    $0.10\pi$ & $0.098\pi$ \\
    $0.15\pi$ & $0.152\pi$ \\
    $0.20\pi$ & $0.201\pi$ \\
    $0.25\pi$ & $0.249\pi$ \\
    $0.30\pi$ & $0.293\pi$ \\
    $0.35\pi$ & $0.350\pi$ \\
    $0.40\pi$ & $0.404\pi$ \\
    $0.45\pi$ & $0.445\pi$ \\
    $0.50\pi$ & $0.500\pi$ \\
    \hline
\end{tabular}
\end{center}

\begin{figure}[H]
    \centering
    \includegraphics[width=0.8\textwidth]{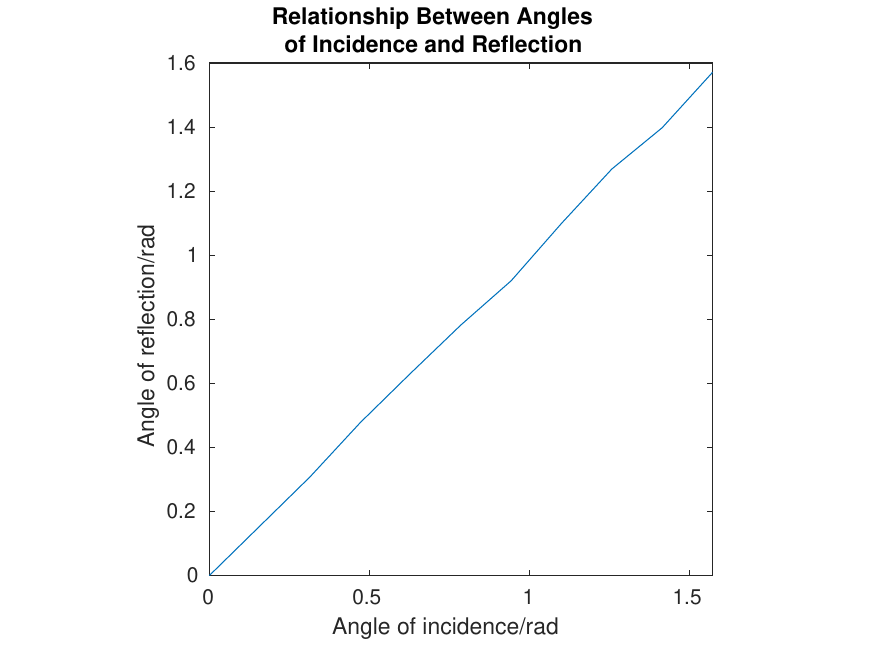}
    \caption{Testing whether or not the law of reflection holds}
\end{figure}

We confirm that the line is $y = x$, up to measurement error, and this is very good evidence that our solution waves satisfy the law of reflection.

\subsection{Shadows}

We conclude this section by investigating the shadow region. To minimise the effects of diffraction and get a representative shadow, we will use the same L shaped region that Trefethen and Gopal used in \cite{doi:10.1073/pnas.1904139116}. Here we use a solution with a maximum error on the boundary of $8 \cdot 10^{-11}$ as we need high accuracy to explore values that are close to 0.

\begin{figure}[H]
     \centering
     \begin{subfigure}[t]{0.49\textwidth}
         \centering
         \includegraphics[width=\textwidth]{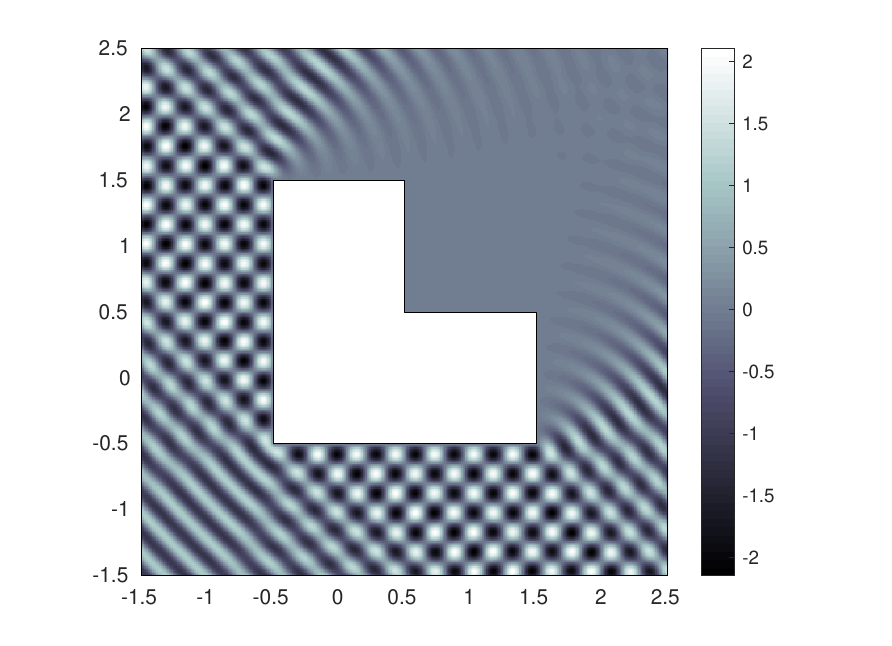}
         \caption{The L shaped region}
     \end{subfigure}
     \begin{subfigure}[t]{0.49\textwidth}
         \centering
         \includegraphics[width=\textwidth]{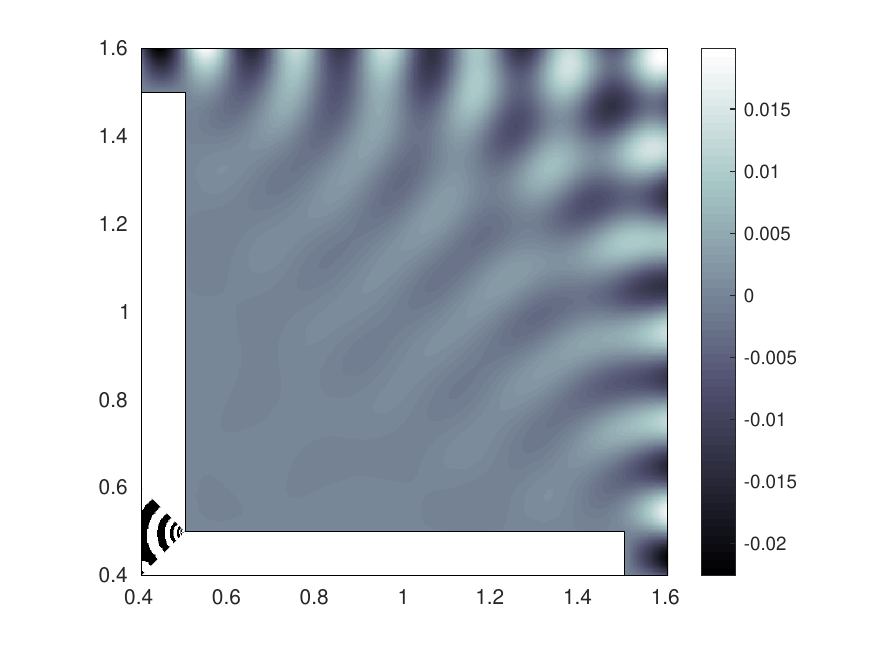}
         \caption{Part of the shadow region}
         \label{the shadow region}
     \end{subfigure}
        \caption{The example region we will use. We note that the discrepancy in the lower left of \ref{the shadow region} is due to a bug in MATLAB's \cite{MATLAB} export to eps feature, and we request the reader to ignore it.}
\end{figure}

\begin{figure}[H]
     \centering
     \begin{subfigure}[t]{0.49\textwidth}
         \centering
         \includegraphics[width=\textwidth]{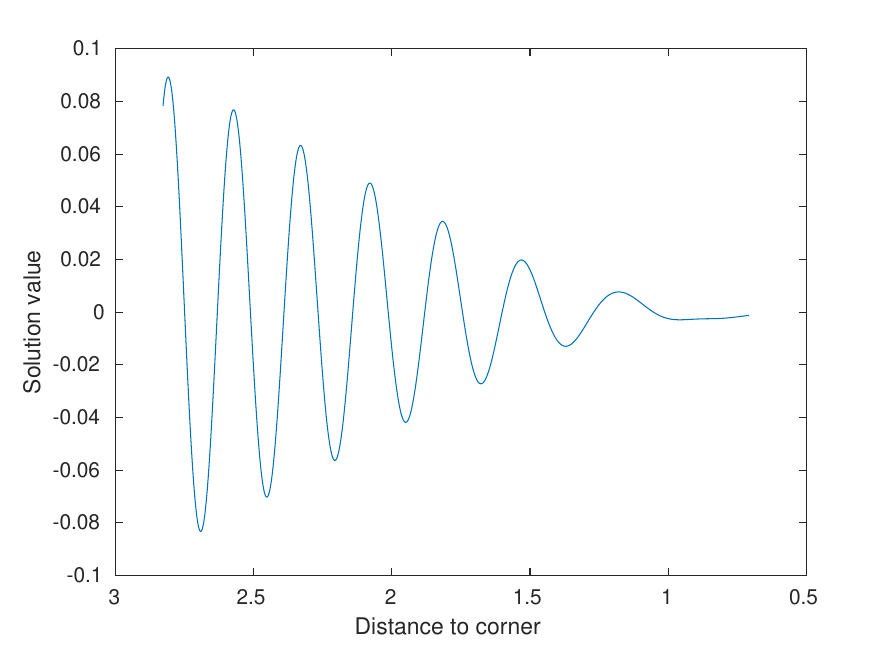}
         \caption{Away from the corner}
     \end{subfigure}
     \begin{subfigure}[t]{0.49\textwidth}
         \centering
         \includegraphics[width=\textwidth]{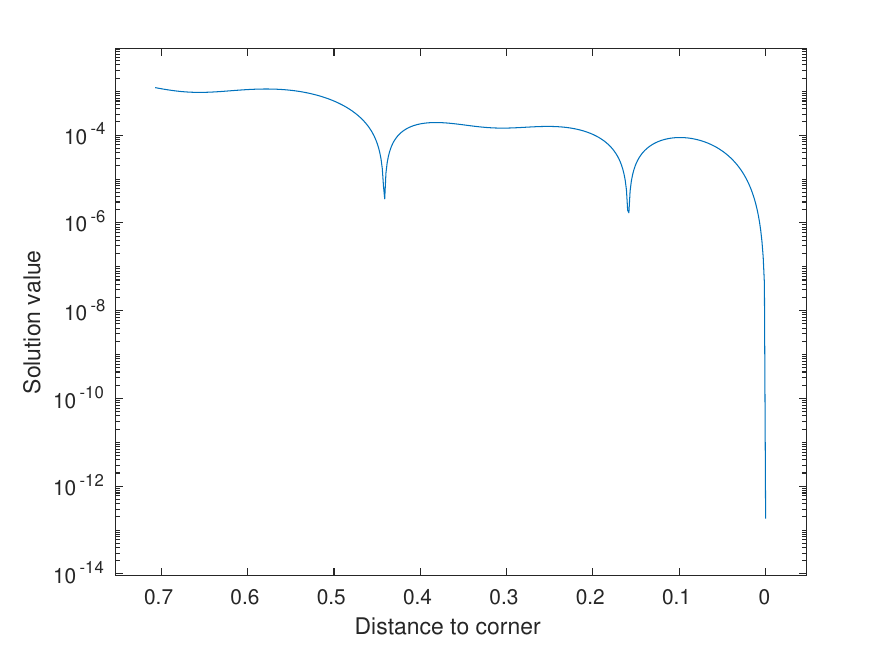}
         \caption{Near the corner}
     \end{subfigure}
     \caption{Decay into the shadow region.}
\end{figure}

We will consider points approaching the corner at $\left(\frac{1}{2}, \frac{1}{2}\right)$ along a line at an angle of $\frac{\pi}{4}$. The decay into the shadow starts out linearly, although this breaks down at a distance of around 1.5 away from the corner where it slows down. It then flattens out significantly, and sits at around $10^{-4}$, interestingly not converging to zero intensity. Repeating the experiment with far lower accuracy gives similar results which suggests this is not due to insufficient accuracy of the method. We see that there is a small area where the sign changes, around 0.17 and 0.45 away from the corner, and this can just about be made out in \ref{the shadow region}. We theorize that these effects are due to wave diffraction that hasn't entirely been eliminated by the region geometry.

\section{Discussion}

We start by using our results from section 4 into a more practical and concise guide to setting method choices. It would be slow and unwieldy to perform that level of analysis to each problem under consideration, so here we outline the understanding from it that allow us to streamline the process.

As we learnt, the key property needed from the sample points is that they are sufficient in number, and that they converge to the corners fast enough. For this reason, we initially choose $s = 1000$, $A = 10$, and $B = 4$. Then, for the purpose of speed, we only use 50 poles per corner, and consider a range of values for $p_r$, usually from 0.1 to 3.1 in increments of 0.3. By displaying the error, it should be clear roughly where the optimum lies, and a finer analysis can be done around this optimum. The error profile provides an invaluable tool to use in diagnosing convergence issues. As we saw before, upward spikes at the corners suggest that $p_r$ is too small, and downward spikes indicate that it is too large.

It is at this stage that we find our biggest opportunities for problems, and if the error is around $O(1)$ then we should start scaling back our ambitions on how low an error we can achieve. The first port of call in this case should be to ensure that we do actually have enough sample points, so we repeat the above experiment with $s = 10000$. For such difficult problems, the pole convergence rate is usually lower, so some time can be saved by limiting $p_r$ to 2.2 instead of 3.1. If this does not give any noticeable difference, then we have found that it is safe to assume that the sample points were not the issue, and $s$ can be turned back down to a more sensible number, such as 1000. We note that we can use figure \ref{bad sample points} to diagnose poor sample point choices if we want to be more precise than a brute force fix as we have suggested.

Our next attack is to try using more poles. For large numbers of poles, we recommend checking that they are not being cut off, and one may want to modify our code to output a warning if this happens. As we have seen in figure \ref{poles vs pole rate}, adding more poles tends to reduce our error for a given value of $p_r$, and it makes intuitive sense that this would help. We typically use the best value of $p_r$ from the above experiment, and then test values of $p$ from 50 to 150 in increments of 10 and look to see if there is improvement.

If the problem still doesn't show signs of giving in, then the prognosis is bleaker still. As mentioned earlier in section 4, we do have one more trick up our sleeves. It is at this stage that we employ the tactics of section \ref{alternative solution forms}, and try out higher order m-Newman forms. Due to the significant increase in size of the problem matrix, the number of poles will usually have to be cut back down to around 50 per corner, maybe lower for regions with more corners, and the number of sample points may need to be lowered as well. We usually test $m = 1,2, \cdots, 5$, although it can be worth going higher to $m = 12$ for example. We warn that this can greatly increase the computation times, and this parameter should be kept quite low.

As we said earlier, in our experiments we have found that the error for troubling regions can usually be brought down to $O(10^{-2})$, which is good enough for a plot. If this is not the case, there are a few last hope efforts that can be made to possibly gain a little bit of extra performance. The simplest method is to vary multiple parameters at once, in the hope a slight improvement can be found with some combination, although this is rather computationally expensive. If the error profile shows a particularly stubborn corner, one can try to increase the number of poles just at that corner, and this should allow for higher m-Newman forms to be used, as the size of the problem matrix is smaller. If one is really desperate, one can attempt to fully optimise the number of poles needed at each corner, although we have not explored this, and it would be an area for further research. Testing if the poles can be placed any closer to the corners is another potential avenue, but can be a risky game, as demonstrated in figure \ref{minimum distance vs error}.

In chapter 5, we moved on to look at how well the method managed to capture wave behaviour. We conclude that it was a resounding success, with all tested phenomenon occurring exactly as expected. A topic we leave to further research would be to modify our code to work with multiple regions, and then test the double slit experiment to verify that the familiar diffraction pattern is produced. Doing a preliminary test, we see that the method is fully compatible with multiple regions, achieving 8 digits of accuracy in the below example. Here we have used two internal points.

We echo the comments of Trefethen and Gopal that the \textit{Lightning Method} indeed shows great promise, and is certainly an effective tool for simple regions in applications, up to a moderate level of accuracy.

\begin{figure}[H]
    \centering
    \includegraphics[trim={0 15mm 10mm 15mm}, clip]{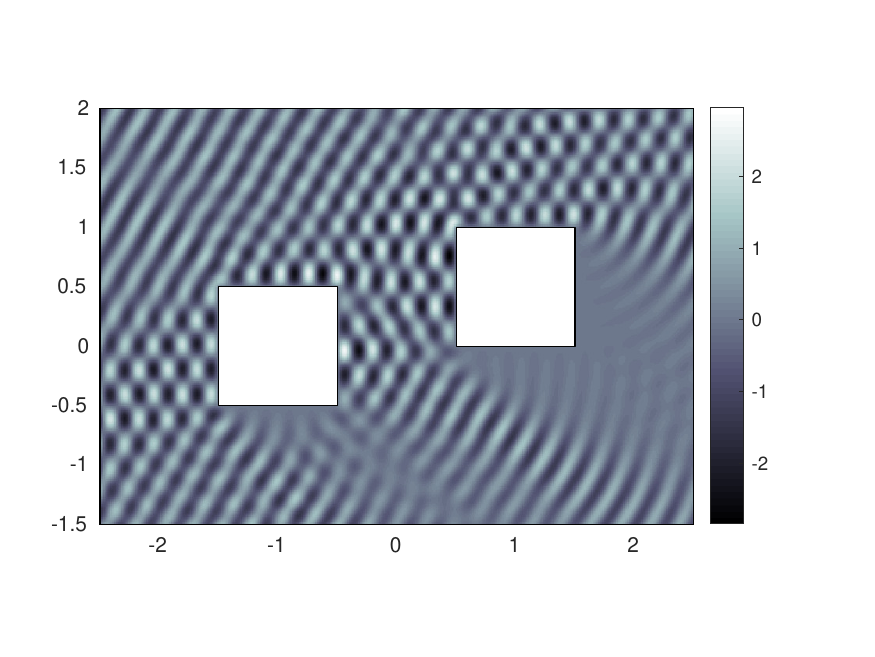}
    \caption{The Lightning Methods works with more than one region.}
\end{figure}

\addcontentsline{toc}{section}{Bibliography}
\bibliographystyle{plain}
\bibliography{Bibliography}

\begin{thebibliography}{10}

\bibitem{BARNETT20087003}
A.H. Barnett and T.~Betcke.
\newblock {Stability and Convergence of the Method of Fundamental Solutions for
  Helmholtz Problems on Analytic Domains}.
\newblock {\em Journal of Computational Physics}, 227(14):7003--7026, 2008.

\bibitem{NIST:DLMF}
{{\it NIST Digital Library of Mathematical Functions}}.
\newblock http://dlmf.nist.gov/, Release 1.1.5 of 2022-03-15.
\newblock F.~W.~J. Olver, A.~B. {Olde Daalhuis}, D.~W. Lozier, B.~I. Schneider,
  R.~F. Boisvert, C.~W. Clark, B.~R. Miller, B.~V. Saunders, H.~S. Cohl, and
  M.~A. McClain, eds.

\bibitem{doi:10.1073/pnas.1904139116}
Abinand Gopal and Lloyd~N. Trefethen.
\newblock {New Laplace and Helmholtz solvers}.
\newblock {\em Proceedings of the National Academy of Sciences},
  116(21):10223--10225, 2019.

\bibitem{doi:10.1137/19M125947X}
Abinand Gopal and Lloyd~N. Trefethen.
\newblock {Solving Laplace Problems with Corner Singularities via Rational
  Functions}.
\newblock {\em SIAM Journal on Numerical Analysis}, 57(5):2074--2094, 2019.

\bibitem{multipole}
D.~Hewett.
\newblock {Multipole Expansions for the Helmholtz Equation}.
\newblock Unpublished note, December 2021.

\bibitem{390536}
Manuel~Pena (https://physics.stackexchange.com/users/147782/manuel pena).
\newblock {Understanding the Sommerfeld radiation condition?}
\newblock Physics Stack Exchange.
\newblock Date accessed: 2022-03-21.

\bibitem{KUPRADZE196482}
V.D. Kupradze and M.A. Aleksidze.
\newblock {The Method of Functional Equations for the Approximate Solution of
  Certain Boundary Value Problems}.
\newblock {\em USSR Computational Mathematics and Mathematical Physics},
  4(4):82--126, 1964.

\bibitem{Diffraction}
Lookangmany.
\newblock Diffraction black and white visualization for w=$\lambda$, October
  2011.
\newblock single frame taken, cropped vertically, URL:
  https://commons.wikimedia.org/w/index.php?curid=16981632, license:
  https://creativecommons.org/licenses/by-sa/3.0. Date accessed: 2022-04-15.

\bibitem{MATLAB}
The Mathworks, Inc., Natick, Massachusetts.
\newblock {\em {MATLAB version 9.11.0.1769968 (R2021b)}}, 2021.

\bibitem{Newman1964RationalAT}
David~J. Newman.
\newblock {Rational Approximation to $|x|$}.
\newblock {\em Journal of Approximation Theory}, 1964.

\bibitem{SCHOT1992385}
Steven~H Schot.
\newblock {Eighty years of Sommerfeld's radiation condition}.
\newblock {\em Historia Mathematica}, 19(4):385--401, 1992.

\bibitem{Sommerfeld1912}
A.~Sommerfeld.
\newblock {Die Greensche Funktion der Schwingungslgleichung}.
\newblock {\em Jahresbericht der Deutschen Mathematiker-Vereinigung},
  21:309--352, 1912.

\bibitem{gimp}
{The GIMP Development Team}.
\newblock {\em GIMP}.
\newblock Website: https://www.gimp.org. Version: 2.10.18.

\bibitem{helmcode}
Lloyd~N. Trefethen.
\newblock helm.m.
\newblock unpublished code, private communication, January 2022.

\bibitem{enwiki:1069719909}
{Wikipedia contributors}.
\newblock Helmholtz equation --- {Wikipedia}{,} the free encyclopedia, 2022.
\newblock [Online; accessed 23-March-2022].

\end{thebibliography}

\appendix
\newpage
\section{Appendix}
Here we give our MATLAB \cite{MATLAB} code, both the Lightning Method class itself, and also the interface we used. The author apologises for areas where the code is not as clean as it should be. We also note an inconsistency in notation, where the sample points on the boundary are referred to as ``bdary pts" instead.

\subsection{Lightning Method Class}
\begin{lstlisting}
% helm.m - Helmholtz solver 17/2/22 - V1.0

% A toolkit for investigating the lightning Helmholtz method, and Helmholtz
% problems with Dirichlet boundary conditions on simply connected domains
% with analytic edges, apart from at a finite number of corners.
% 
% How to use:
%   Each problem takes in input in four parts: region information, PDE
%   information, pole information, and boundary points information. The
%   region information is given on initialisation of a Helmholtz problem
%   object (prob = helm(...)), and the other information is given through
%   the methods, set_problem_data, set_poles, and set_bdary_pts (eg, prob =
%   prob.set_poles(...)). Boundary point information must be set after PDE
%   information.
%
%   All inputs are optional, and when not given, a default example input is
%   used. Each input is given in two parts: a string describing what the
%   input is, and then the next parameter given will be the data for that
%   input (eg ('wave number', 5)). If ordering of parameters is not
%   specified, then any order will do. If an input needs multiple
%   parameters, put them in a cell array (eg ('shape', {'polygon', 3}) for
%   a triangle).
%
%   set_plot_data is an optional input that defines the region over which
%   to plot, and if it is not used then it will be run with default options
%   before plotting the solution. It's use is recommended as it only
%   requires information on the region shape, so if other problem
%   parameters are being changed, the computation doesn't need to be
%   repeated. Similarly with set_error_pts.

classdef helm

properties
    k
    f
    P
    corners
    corners_l
    edges
    edge_lengths
    lengths
    angles
    int_angles
    bdary_pts_plot
    region
    x_lims
    y_lims
    bounding_box
    plot_pts_resolution
    poles
    bdary_pts
    bdary_values
    reflection
    solution
    error_pts
    error
    errors
    conditioning
    output
end

methods

function self = helm(varargin)
% Defining the region upon initialisation of the problem
%
% Inputs:
%   'plot': Will plot the region, takes in false or true. This is intended
%   for checking the region is valid (simply connected, no self
%   intersections) when constructing test regions. Defaults to false.
%   
%   'size': A scale factor for the region. Defaults to 1.
%
%   'shape': Defining the corners of the region.
%   There are five options:
%   1: {'polygon', n}. A regular polygon with n sides
%   2: {'L shape'}. A region in the shape of a capital L
%   3: {'arrow'}. An inverted kite shape
%   4: {'lightning'}. The lightning Helmholtz symbol
%   5: {'sector', theta}. A sector of circle with internal angle theta.
%      Need to select 'sector' option for edges as well
%   6: {'custom', [c_1;...;c_n]}. A list of complex numbers giving the
%      vertices of the shape going anticlockwise
%   Defaults to {'polygon', 4} (a square)
%
%   'edges': Defining how the vertices of the region should be connected.
%   Warning: the program will not automatically detect if it is a simply
%   connected domain, and some choices will cause the domain to be
%   invalid. It is recommended to plot the shape as a check if you are not
%   using straight lines.
%   There are four options:
%   1: {'straight'}. Connecting the corners with straight lines
%   2: {'curved', s}. Circular arcs. -1 is a circular arc cutting into the
%      region, 1 is a circular arc bulging outside the region, and as it
%      approaches 0, it is a straight line (although will break at 0).
%      s is a continuous parameter and works for all non zero values.
%   3: {'wiggly', [amplitude, wiggle count]}. Adds a sin curve of the form
%      amplitude*sin(wiggle_count*pi*t) to the edge to give it wiggles.
%   4: {'sector'}. Connects the corners of a 'sector' shape with the
%      appropriate edges
%   5: {'custom', {e_1;...;e_n}}. Each e_i is an anonymous function that
%      connects the i'th corner at t=0 to the i+1st corner at t=1. It is
%      recommended to parameterise by arc length, at least roughly.
%   Defaults to {'straight'}
%
%   'bdary resolution': Defining the minimum gap between points on the
%   boundary. These points will be used for plotting the region, and also
%   computing where the line of poles at each corner will intersect the
%   boundary. When the domain has edges with rapid changes, a higher
%   resolution will be required. Region is approximated by a polynomial.
%   Defaults to 0.01

    % Parsing the input from the user
    [self, plot_bool, bdary_resolution] = parse_region_input(...
        self, varargin);
    
    % Argument of interior bisectors, and size of interior angles
    [self.angles, self.int_angles] = self.get_angles();
    self = self.get_edge_lengths();
    self = self.get_region(bdary_resolution);

    % Plotting if that option is true
    if plot_bool
        self.plot_region('bdary pts dense', true)
    end
    
    % Maximum lengths of interior bisectors until it hits another edge
    self.lengths = self.get_lengths();
end

function [self, plot_bool, bdary_resolution] = parse_region_input(...
        self, varargin)
    % Setting default values
    plot_bool = false;
    size = 1;
    corner_input = {'polygon', 4};
    edge_input = {'straight'};
    bdary_resolution = 0.01;
    
    % Parsing user input
    for i = 1:2:length(varargin{1})
        key = varargin{1}{i}; value = varargin{1}{i+1};
        if strcmp(key, 'plot')
            plot_bool = value;
        elseif strcmp(key, 'size')
            size = value;
        elseif strcmp(key, 'shape')
            corner_input = value;
        elseif strcmp(key, 'edges')
            edge_input = value;
        elseif strcmp(key, 'bdary resolution')
            bdary_resolution = value;
        end
    end
    
    % Computing corners and edges and updating the object
    self = get_corners(self, size, corner_input);
    self = get_edges(self, edge_input);
end

function self = get_corners(self, size, corner_data)
    if strcmp(corner_data{1}, 'polygon')
        n = corner_data{2};
        starting_angle = -pi*(1/n+1/2);
        self.corners = exp(starting_angle*1i+(0:n-1)*2*pi*1i/n).';
    elseif strcmp(corner_data{1}, 'L shape')
        self.corners = [0; 2; 2+1i; 1+1i; 1+2i; 2i]-0.5-0.5i;
    elseif strcmp(corner_data{1}, 'arrow')
        self.corners = [1i; -1-1i; -0.5i; 1-1i];
    elseif strcmp(corner_data{1}, 'lightning')
        self.corners = [0; 1+2i; 0.5+1.85i; 1+3i; 2.7i; -0.2+1.3i;...
            0.1+1.4i]*exp(-0.35i)-0.375-1.5i;
    elseif strcmp(corner_data{1}, 'sector')
        theta = corner_data{2}/2;
        self.corners = [exp(-1i*theta); 0; exp(1i*theta)]+0.1;
    elseif strcmp(corner_data{1}, 'custom')
        self.corners = corner_data{2};
    end
    self.corners = size*self.corners;
    self.corners_l = [self.corners(end); self.corners; self.corners(1)];
    self.P = length(self.corners);
end

function self = get_edges(self, edge_data)
    if strcmp(edge_data{1}, 'straight')
        self.edges = get_edges_straight(self);
    elseif strcmp(edge_data{1}, 'curved')
        self.edges = get_edges_curved(self, edge_data{2});
    elseif strcmp(edge_data{1}, 'wiggly')
        self.edges = get_edges_wiggly(self, edge_data{2});
    elseif strcmp(edge_data{1}, 'sector')
        self.edges = get_edges_sector(self);
    elseif strcmp(edge_data{1}, 'custom')
        self.edges = edge_data{2};
    end
    % It will be useful to have the first edge repeated at the end
    self.edges{end+1} = self.edges{1};
end

function edges = get_edges_straight(self)
    for i = 2:self.P+1
        z = self.corners_l(i)-self.corners_l(i-1);
        curve = @(t) self.corners_l(i-1) + t*z;
        edges{i-1} = curve;
    end
end

function edges = get_edges_curved(self, s)
    % Z(t) is a circle that goes through 0 at t=0, and 1 at t=1. As s
    % goes to 0, it approaches a straight line, and s=1 is a
    % semicircle. If the input is negative, the arc will be below the
    % real axis. Z(t) is then passed through an affine map so it
    % connects the two corners

    % sign controls whether it is above or below the axis
    if s < 0; sign = -1; s = -s; % Bulging out
    else; sign = 1;              % Cutting in
    end

    % Defining constants (depending on s) to construct Z(t)
    s_scale = 1/s-1; % Anything cont that maps 0 to inf and 1 to 0
    R = sqrt(1+s_scale^2)/2; % The radius
    % Affine map constants so we map [0, 1] to the curve
    t_0 = asin(s_scale/(2*R));
    t_scale = 2*t_0 - pi;
    t_shift = pi-t_0;
    Z = @(t) R*exp(-sign*1i*(t_scale*t+t_shift))+1/2+sign*1i*s_scale/2;

    % Constructing the edges
    edges = cell(1, self.P);
    for j = 2:self.P+1
        % Defining the constants used in the affine map
        z = self.corners_l(j)-self.corners_l(j-1);
        Z_scale = abs(z)*exp(1i*angle(z));
        Z_shift = self.corners_l(j-1);
        curve = @(t) Z(t)*Z_scale+Z_shift;
        edges{j-1} = curve;
    end
end

function edges = get_edges_wiggly(self, s)
    % We add a sin curve to make the line wiggly
    edges = cell(1, self.P);
    wiggle_amplitude = s(1); wiggle_count = s(2);
    wiggle = @(t) t + 1i*wiggle_amplitude*sin(wiggle_count*pi*t);
    for j = 2:self.P+1
        z = self.corners_l(j)-self.corners_l(j-1);
        curve = @(t) self.corners_l(j-1)+wiggle(t)*z;
        edges{j-1} = curve;
    end
end

function edges = get_edges_sector(self)
    scale = 2*(pi-angle(self.corners(3)-0.1));
    offset = angle(self.corners(3)-0.1);
    edge_1 = @(t) exp(1i*scale*t+1i*offset)+0.1;
    edge_2 = @(t) (self.corners(1)-0.1)*(1-t)+0.1;
    edge_3 = @(t) (self.corners(3)-0.1)*t+0.1;
    edges = {edge_1, edge_2, edge_3};
end

function [angles, int_angles] = get_angles(self)
    angles = zeros(1, self.P); int_angles = zeros(1, self.P);
    eps = 10^-10;
    for i = 1:self.P
        tan_1 = (self.edges{i+1}(eps)-self.edges{i+1}(0))/eps;
        tan_2 = (self.edges{i}(1-eps)-self.edges{i}(1))/eps;
        interior_angle = acos(real(tan_1*tan_2')/abs(tan_1*tan_2));
        reflex = imag(tan_1'*tan_2);
        if reflex < 0
            interior_angle = 2*pi - interior_angle;
        end
        interior_bisector = angle(tan_2)-interior_angle/2;
        angles(i) = interior_bisector; int_angles(i) = interior_angle;
    end
end

function self = set_plot_data(self, varargin)
% Defining the rectangular region to be plotted over, and finding a mesh of
% plotting points.
%
% Inputs
%   'edges': Takes in a string or cell array (if a parameter needs to be
%   provided) with the options that define the box. The first entry of the
%   cell is the name of the option, and the second is the parameter.
%   Defaults to 'buffer'. If a parameter is not given then the default is
%   used.
%       'buffer': Each edge of the rectangle will be a fixed radius away
%       from the closest point. Defaults to 1
%
%       'edge positions': Takes in either a scalar or a vector that
%       describe the x coordinates of the left and right sides, and the y
%       coordinates of the lower and upper sides. If a scalar, C, is given,
%       the vector equivalent will be [-C, C, -C, C]. Defaults to 1.5
%
%   'square': Determines whether or not the region is forced to be a
%   square. Takes in either true, false, 'max', or 'min'.
%   If true is given then it will default to max. The region will in
%   general be rectangular, and the max/min options decide whether or not
%   to take the larger or smaller side length to be side length of the
%   square. Defaults to 'max'
%
%   'plot pts resolution': Defining the spacing of the grid of external
%   points. Defaults to 0.01

    % Parsing the input from the user
    [self, edges_settings, square_plot, square_setting] ...
    = parse_plot_data_input(self, varargin);

    % Computing the edges of the bounding rectangle
    self = get_bounding_edges(self, edges_settings);
    if square_plot
        self = make_bounding_box_square(self, square_setting);
    end
    self = get_bounding_box(self);

    function [self, edges_settings, square_plot, square_setting] ...
             = parse_plot_data_input(self, varargin)
        % Setting default values
        buffer_setting = 1;
        edge_positions_settings = get_edge_positions_vector(1.5);
        edges_settings = {'buffer', buffer_setting};
        square_plot = true; square_setting = 'max';
        self.plot_pts_resolution = 0.01;

        % Updating values if they have been provided
        for i = 1:2:length(varargin{1})
            key = varargin{1}{i}; value = varargin{1}{i+1};
            if strcmp(key, 'edges')
                edges_settings = get_edge_settings(value,...
                    buffer_setting, edge_positions_settings);
            elseif strcmp(key, 'square')
                [square_plot, square_setting] ...
                    = get_square_settings(value, square_setting);
            elseif strcmp(key, 'plot pts resolution')
                self.plot_pts_resolution = value;
            end
        end
    end

    function edge_positions = get_edge_positions_vector(x)
        if isscalar(x)
            edge_positions = [-x, x, -x, x];
        else
            edge_positions = x;
        end
    end

    function edge_settings = get_edge_settings(input, ...
             buffer_setting, edge_positions_settings)
        if isscalar(input)
            edge_method = input;
            if strcmp(edge_method, 'buffer')
                edge_settings = {edge_method, buffer_setting};
            elseif strcmp(edge_method, 'edge positions')
                edge_settings = {edge_method, edge_positions_settings};
            end
        else
            if strcmp(input{1}, 'buffer')
                edge_settings = input;
            elseif strcmp(input{1}, 'edge positions')
                setting = get_edge_positions_vector(input{2});
                edge_settings = {input{1}, setting};
            end
        end
    end

    function [square_plot, square_setting] = get_square_settings(input,...
            square_setting_default)
        if islogical(input)
            square_plot = input;
            square_setting = square_setting_default;
        else
            square_plot = true;
            square_setting = input;
        end
    end

    function self = get_bounding_edges(self, edge_settings)
        if strcmp(edge_settings{1}, 'buffer')
            self = get_bounding_edges_buffer(self, edge_settings{2});
        elseif strcmp(edge_settings{1}, 'edge positions')
            self = get_bounding_edges_positions(self, edge_settings{2});
        end
    end

    function self = get_bounding_edges_buffer(self, buffer)
        [x_lim, y_lim] = boundingbox(self.region);
        x_lim_left = x_lim(1) - buffer;
        x_lim_right = x_lim(2) + buffer;
        y_lim_bottom = y_lim(1) - buffer;
        y_lim_top = y_lim(2) + buffer;
        self.x_lims = [x_lim_left, x_lim_right];
        self.y_lims = [y_lim_bottom, y_lim_top];
    end

    function self = get_bounding_edges_positions(self, edge_positions)
        self.x_lims = [edge_positions(1), edge_positions(2)];
        self.y_lims = [edge_positions(3), edge_positions(4)];
    end

    function self = make_bounding_box_square(self, square_setting)
        width = self.x_lims(2) - self.x_lims(1);
        height = self.y_lims(2) - self.y_lims(1);
        correction = abs(width - height)/2;
        if width > height
            if strcmp(square_setting, 'max')
                self.y_lims(1) = self.y_lims(1) - correction;
                self.y_lims(2) = self.y_lims(2) + correction;
            else
                self.x_lims(1) = self.x_lims(1) + correction;
                self.x_lims(2) = self.x_lims(2) - correction;
            end
        elseif width < height
            if strcmp(square_setting, 'max')
                self.x_lims(1) = self.x_lims(1) - correction;
                self.x_lims(2) = self.x_lims(2) + correction;
            else
                self.y_lims(1) = self.y_lims(1) + correction;
                self.y_lims(2) = self.y_lims(2) - correction;
            end
        end
    end

    function self = get_bounding_box(self)
        size_x = self.x_lims(2) - self.x_lims(1);
        size_y = self.y_lims(2) - self.y_lims(1);
        self.bounding_box = {'Position', ...
             [self.x_lims(1), self.y_lims(1), size_x, size_y]};
    end
end

function self = get_edge_lengths(self)
    % We find the approximate arc length via sampling the boundary at a
    % fine mesh, and approximating the curve as piecewise linear
    N_fine = 100;
    difference_matrix = get_difference_matrix(N_fine);
    self.edge_lengths = zeros(self.P, 1);
    for i = 1:self.P
        edge_curve = self.edges{i};
        bdary_mesh = edge_curve(linspace(0, 1, N_fine));
        differences = difference_matrix*bdary_mesh';
        distances = abs(differences);
        edge_length = sum(distances);
        self.edge_lengths(i) = edge_length;
    end

    function difference_matrix = get_difference_matrix(size)
        % (n-1)*n matrix with 1 on superdiagonal and -1 on subdiagonal.
        % When right multiplied, finds the differences between elements
        sub_diag = -1*eye(size-1); super_diag = eye(size-1);
        zero_col = zeros(size-1, 1);
        sub_diag = [sub_diag zero_col];
        super_diag = [zero_col super_diag];
        difference_matrix = sub_diag+super_diag;
    end
end

function self = get_region(self, bdary_resolution)
    % We add points on the boundary so that along the edge, there is a
    % point approximately within the tolerance distance as measured
    % along the edge length. Custom edge parameterisations should be
    % parametrised by arc length for this to be exact.
    self.bdary_pts_plot = self.get_even_spaced_pts(bdary_resolution);
    X_plot = real(self.bdary_pts_plot);
    Y_plot = imag(self.bdary_pts_plot);
    % If the region has straight edges, we have too many bdary pts
    warning('off', 'MATLAB:polyshape:repairedBySimplify')
    self.region = polyshape(X_plot, Y_plot);
end

function self = set_problem_data(self, varargin)
% Sets information about the PDE itself - the wave number and the boundary
% data. If both inputs are given, they must be done in this order.
%
% Inputs:
%   'k': the wave number. This is the k in laplace(u) + k^2*u = 0. Defaults
%   to 30
%
%   'boundary data': This has two parts to it. The first part is the type,
%   and should be one of three different strings: 'plane wave', 'point
%   source', and 'custom'. In the first two cases, the second entry is a
%   complex parameter that appears in the function, and for custom data, it
%   is an anonymous function (this must be a solution of the Helmholtz
%   equation). Input given as a cell array, {part_1, part_2}. Defaults to
%   {'plane wave', -1i*5*pi/6}

    % Setting default values
    self.k = 30;
    self.f = @(z) exp(-1i*real(self.k*exp(-1i*5*pi/6)*z));

    % Parsing user input
    for i = 1:2:length(varargin)
        key = varargin{i}; value = varargin{i+1};
        % Wave number
        if strcmp(key, 'k')
            self.k = value;
            self.f = @(z) exp(-1i*real(self.k*exp(-1i*5*pi/6)*z));
        % Boundary data
        elseif strcmp(key, 'boundary data')
            data_type = value{1}; data_value = value{2};
            if strcmp(data_type, 'plane wave')
                self.f = @(z) exp(-1i*real(self.k*exp(data_value)*z));
            elseif strcmp(data_type, 'point source')
                self.f = @(z) besselh(0, self.k*abs(z-data_value));
            elseif strcmp(data_type, 'custom')
                self.f = data_value;
            end
        end
    end
end

function self = set_poles(self, varargin)
% Computes the position of the poles
%
% Inputs:
%   'poles': A column vector where the i'th element gives the number of
%   poles at the i'th corner. If a scalar is given, it will be converted
%   to a column vector with all the same entries. Defaults to 10
%   
%   'line length max': How long the line of poles approaching a corner is
%   allowed to be. The length of the line of poles will be given by the
%   minimum of this value, and the length limit imposed by the region
%   shape (potentially modifed by other parameters). Defaults to 1
%
%   'region limit ratio': The line length is limited by the other side of
%   the region. This length will be multiplied by this parameter. Defaults
%   to 0.8
%
%   'pole exp rate': The convergence constant that controls how fast the
%   poles approach the corner. It is the C in exp(-C*sqrt(n)). Defaults to
%   pi
%
%   'cutoff': The minimum distance that the poles are allowed to get to the
%   corners before they are removed. Defaults to 10^-9

    % Parsing the input from the user
    [n_pole, line_length_max, region_limit_ratio, pole_exp_rate, cutoff]...
        = parse_pole_inputs(varargin);

    % Getting the location of the poles
    self = compute_poles(self);

    function [n_pole, line_length_max, region_limit_ratio, pole_exp_rate,...
             cutoff] = parse_pole_inputs(varargin)
        % Setting default values
        n_pole = get_arr(10, self.P);
        line_length_max = 1;
        region_limit_ratio = 0.8;
        pole_exp_rate = pi;
        cutoff = 10^(-9);

        % Updating values if they have been provided
        for i = 1:2:length(varargin{1})
            key = varargin{1}{i}; value = varargin{1}{i+1};
            if strcmp(key, 'poles')
                n_pole = get_arr(value, self.P);
            elseif strcmp(key, 'line length max')
                line_length_max = value;
            elseif strcmp(key, 'region limit ratio')
                region_limit_ratio = value;
            elseif strcmp(key, 'pole exp rate')
                pole_exp_rate = value;
            elseif strcmp(key, 'cutoff')
                cutoff = value;
            end
        end
    end

    function self = compute_poles(self)
        self.poles = [];
        for i = 1:self.P
            n = n_pole(i);
            exp_cluster = exp(-(0:n-1)*pole_exp_rate/sqrt(n));
            region_length = self.lengths(i)*region_limit_ratio;
            length = min(line_length_max, region_length);
            distances = exp_cluster*length;
            distances = distances(distances > cutoff);
            poles_i = (self.corners(i)+distances*exp(1i*self.angles(i))).';
            self.poles = [self.poles; poles_i];
        end
        self.poles = unique(self.poles);
    end
end

function lengths = get_lengths(self)
    % Finding how far away the nearest edge is along the interior
    % bisector from each corner. Used to ensure no exterior poles
    lengths = zeros(1, self.P);
    for j = 1:self.P
        corner = self.corners(j);
        angle_j = self.angles(j);
        lengths(j) = get_length(self);
    end

    function length_i = get_length(self)
        % Define a line in the direction of the interior bisector that is
        % at least as long as the region, finding the intersection points
        % with the region and the line, picking the closest one, and
        % computing the length between it and the corner
        max_length = perimeter(self.region) / 2;
        bisector_end = corner + max_length*exp(1i*angle_j);
        line_x_pts = [real(corner), real(bisector_end)];
        line_y_pts = [imag(corner), imag(bisector_end)];
        [bdary_x, bdary_y] = boundary(self.region);
        [intersections_x, intersections_y] = polyxpoly(...
            bdary_x, bdary_y, line_x_pts, line_y_pts);
        intersections = intersections_x + 1i*intersections_y;
        distances = abs(intersections - corner);
        distances = distances(distances > 10^-10);
        length_i = min(distances);
    end
end

function self = set_bdary_pts(self, varargin)
% Computes the location of points on the boundary, and evaluates the
% boundary data on them - these points will be used in the Runge part of
% the expansion of the solution evaluated on. The points are associated to
% a corner (not an edge), so the points for one corner will start halfway
% along one edge, and end halfway along the next edge.
%
% Inputs
%   'bdary pts': Takes in a positive integer scalar or a positive integer
%   vector, and gives half the number of points associated with a corner (n
%   on each side). Default is 100
%
%   'distribution': How the points are distributed along the edges. Each
%   input should be given as a cell array. Note that if multiple
%   distributions are superimposed, due to removing duplicates and
%   requiring an integer number of points from each distribution, the total
%   number of bdary pts might not match the desired amount.
%   Distributions are described as continuous monotonically decreasing
%   anonymous functions where 0 maps to 0, and 1 maps to 1. The 0 end is at
%   the corner, and the 1 end is halfway along the edge. Default is
%   {'exponential', 4, 4, 1}.
%   1: 'equispaced': points are evenly spread
%   2: 'exponential': The distribution is given by t^A*e^(-B*j^C) where A,
%   B and C are constants
%   3: 'mixed': this is a combination of 'equispaced' and 'exponential'.
%   The first parameter is the ratio of how the points should be split up
%   between the exponential part and the equispaced part (eg, an input of 2
%   means a 2:1 ratio), and the next two are same parameters as in the
%   'exponential' option.
%   4: 'custom': The anonymous function is given directly. If multiple are
%   given they will be superimposed, and a list at the end giving the ratio
%   of how the bdary pts are to be distributed between them needs to be
%   given. Inputs should be given in a cell array

    % Parsing the input from the user
    [n_bdary_pts, dist_setting] = parse_inputs_bdary_pts(self, varargin);
    
    % dist_funs is a list of functions where 0 maps very close to 0, and 1
    % maps to 1. Stored in a 2xn cell array, where the first row are the
    % functions, and the second row is the proportion of the bdary points
    % computed using that function.
    dist_funs = get_dist_functions();
    self = get_bdary_pts(self);

    function [n_bdary_pts, dist_setting]...
            = parse_inputs_bdary_pts(self, varargin)
        % Setting default values
        n_bdary_pts = get_arr(100, self.P);
        dist_setting = {'exponential', 4, 4, 1};

        % Updating values if they have been provided
        for i = 1:2:length(varargin{1})
            key = varargin{1}{i}; value = varargin{1}{i+1};
            if strcmp(key, 'bdary pts')
                n_bdary_pts = get_arr(value, self.P);
            elseif strcmp(key, 'distribution')
                dist_setting = value;
            end
        end
    end

    function dist_funs = get_dist_functions()
        dist_type = dist_setting{1};
        if strcmp(dist_type, 'equispaced')
            dist_funs = get_equispaced();
        elseif strcmp(dist_type, 'exponential')
            dist_funs = get_expon(dist_setting{2}, dist_setting{3}, ...
                dist_setting{4});
        elseif strcmp(dist_type, 'mixed')
            dist_funs = get_mixed(dist_setting{2}, dist_setting{3}, ...
                dist_setting{4});
        elseif strcmp(dist_type, 'custom')
            dist_funs = get_custom(dist_setting{2});
        end
    end

    function dist_fun = get_equispaced()
        fun = @(t) 1-t;
        dist_fun = {fun; 1};
    end

    function dist_fun = get_expon(A, B, C)
        fun = @(t) t.^A.*exp(-B*(1-t).^C);
        dist_fun = {fun; 1};
    end

    function dist_fun = get_mixed(ratio, A, B)
        equi_proportion = 1/(ratio+1);
        expon_proportion = ratio/(ratio+1);
        fun_equi = @(t) 1-t;
        fun_expon = @(t) t.*exp(-A*(1-t).^B);
        dist_fun = {fun_equi, fun_expon; ...
                    equi_proportion, expon_proportion};
    end

    function dist_funs = get_custom(dist_settings)
        if length(dist_settings) == 1
            dist_funs = {dist_settings; 1};
        else
            proportions = dist_settings{end};
            proportions = proportions/sum(proportions);
            proportions = num2cell(proportions);
            dist_funs = {dist_settings{1:end-1}; proportions{:}}; 
        end
    end
    
    function self = get_bdary_pts(self)
        self.bdary_pts = [];
        for fun_data = dist_funs
            fun = fun_data{1};
            proportion = fun_data{2};
            for j = 1:self.P
                n = ceil(n_bdary_pts(j)*proportion);
                edge_1 = self.edges{j}; edge_2 = self.edges{j+1};
                bdary_pts_corner = get_bdary_pts_corner(n, edge_1, edge_2, fun);
                self.bdary_pts = [self.bdary_pts bdary_pts_corner];
            end
        end
        self = process_bdary_pts(self);
    end
    
    function bdary_pts_corner = get_bdary_pts_corner(n, edge_1, edge_2, fun)
        % Returns the bdary points associated with a particular corner
        T = linspace(0, 1, n);
        T = T(2:end);
        % Mapping T to the desired distribution
        T = fun(T);
        % T needs to be mapped so the points approach the corner from
        % halfway along the edge from each side
        bdary_pts_1 = edge_1(1-T/2);
        bdary_pts_2 = edge_2(T/2);
        bdary_pts_corner = [bdary_pts_1 bdary_pts_2];
    end

    function self = process_bdary_pts(self)
        self.bdary_pts = round(self.bdary_pts, 12).';
        self.bdary_pts = unique(self.bdary_pts);
        self.bdary_pts = sort(self.bdary_pts);
        self.bdary_values = self.f(self.bdary_pts);
    end
end

function even_spaced_pts = get_even_spaced_pts(self, resolution)
    % We add points on the boundary so that along the edge, there is a
    % point approximately within the resolution as measured along the edge
    % length. Custom edge parameterisations should be parameterised by arc
    % length for this to be exact.
    even_spaced_pts = [];
    for i = 1:self.P
        edge_length = self.edge_lengths(i);
        T = linspace(0, 1, floor(edge_length/resolution)+1)';
        new_pts = self.edges{i}(T);
        new_pts(end) = [];
        even_spaced_pts = [even_spaced_pts; new_pts];
    end
end

function self = set_error_pts(self, varargin)
    resolution = 0.01;
    if ~isempty(varargin)
        resolution = varargin{1};
    end
    self.error_pts = self.get_even_spaced_pts(resolution);
end

function self = solve(self, varargin)
% Computes an anonymous function that gives the solution
%
% Inputs:
%   'newman': This is the highest order oh Hankel functions included in the
%   Newman part of the form of solution. Defaults to 1
%
%   'runge': This is the number of terms in the Runge part of the form of
%   solution. Defaults to 40
%
%   'runge extended': Determines whether or not the term corresponding to
%   negative indices are included in the series. Takes in a boolean.
%   Defaults to false
%
%   'conditioning': Controls whether or not the conditioning number of the
%   problem matrix is computed or not. Takes in a boolean. Defaults to
%   false.

    [n_newman, n_runge, runge_extended, con] = parse_solve_input(varargin);
    problem_matrix = self.get_problem_matrix(...
        self.bdary_pts, n_newman, n_runge, runge_extended);
    col_scale = get_col_scale();
    coefficients = (problem_matrix*col_scale)\self.bdary_values;
    scaled_coefficients = col_scale*coefficients;
    self.reflection = @(z) self.get_problem_matrix(...
        z, n_newman, n_runge, runge_extended)*scaled_coefficients;
    self.solution = @(z) self.reflection(z) - self.f(z);
    self = set_error(self);
    
    if con == true
        self.conditioning = cond(problem_matrix);
    end

    function [n_newman, n_runge, runge_extended, con] = ...
            parse_solve_input(varargin)
        % Setting default values
        n_newman = 1;
        n_runge = 40;
        runge_extended = false;
        con = false;

        % Updating values if they have been provided
        for i = 1:2:length(varargin{1})
            key = varargin{1}{i}; value = varargin{1}{i+1};
            if strcmp(key, 'newman')
                n_newman = value;
            elseif strcmp(key, 'runge')
                n_runge = value;
            elseif strcmp(key, 'runge extended')
                runge_extended = value;
            elseif strcmp(key, 'conditioning')
                con = value;
            end
        end
    end

    function col_scale = get_col_scale()
        % Making a sparse square matrix to scale the columns of the problem
        % matrix so they have 2 norm of 1 - better conditioning
        col_scale = 1./vecnorm(problem_matrix);
        matrix_size = length(col_scale);
        col_scale = spdiags(col_scale', 0, matrix_size, matrix_size);
    end

    function self = set_error(self)
        if isempty(self.error_pts)
            self = self.set_error_pts();
        end
        true_solution = self.f(self.error_pts);
        computed_solution = self.reflection(self.error_pts);
        self.errors = abs(true_solution - computed_solution);
        self.error = norm(self.errors, 'inf');
    end
end

function problem_matrix = get_problem_matrix(self, ...
        input_pts, n_newman, n_runge, runge_extended)
    % This is the matrix that describes the problem. On the left is the
    % Newman part (deals with the corners of the region via Hankel
    % function expansion around the poles), and on the right is the
    % Runge part (the smooth part of the problem). All entries are of
    % the form H_j(k*|z|)*exp(i*j*arg(z)), for j=0 or 1
    %
    % Each row corresponds to one of the points in input_pts. For the
    % Newman part, we have a 0th order part, and a first order part. In
    % each part, each column corresponds to a pole, and the (i, j) entry is
    % a function of p_j - z_i where p_j is the j'th pole, and z_i is the
    % i'th point in input_pts. For the Runge part, we expand around one
    % internal point, and the j'th column corresponds to the j'th order
    % Hankel function. There are R+1 terms in the expansion

    newman_part = get_newman_part();
    runge_part = get_runge_part();
    problem_matrix = [newman_part runge_part];

    function newman_part = get_newman_part()
        eval_points = input_pts - self.poles.';
        eval_abs = abs(eval_points);
        eval_angles = eval_points./eval_abs;
        newman_part = besselh(0, self.k*eval_abs);
        for n = 1:n_newman
            newman_part = [newman_part ...
                self.k*besselh(n, self.k*eval_abs).*eval_angles.^n];
        end
    end

    function runge_part = get_runge_part()
        int_point = 0;
        points = input_pts - int_point;
        abs_points = abs(points); arg_points = angle(points);
        if runge_extended == false
            runge_part = get_runge();
        else
            runge_part = get_runge_extended();
        end

        function runge = get_runge()
            runge = zeros(length(input_pts), n_runge+1);
            for j = 0:n_runge
                runge(:,j+1) = besselh(j, self.k*abs_points)...
                    .*exp(1i*j*arg_points);
            end
        end
    
        function runge = get_runge_extended()
            runge = zeros(length(input_pts), 2*n_runge+1);
            runge(:,1) = besselh(0, self.k*abs_points);
            for j = 1:n_runge
                runge(:,j+1) = besselh(j, self.k*abs_points)...
                    .*exp(1i*j*arg_points);
                runge(:,j+m+1) = besselh(j, self.k*abs_points)...
                    .*exp(-1i*j*arg_points);
            end
        end
    end
end

function solution_plot = get_solution_plot(self, solution_type, range)
    [grid_x, grid_y] = meshgrid(...
                self.x_lims(1):self.plot_pts_resolution:self.x_lims(2),...
                self.y_lims(1):self.plot_pts_resolution:self.y_lims(2));
    plot_pts = grid_x + 1i*grid_y;
    plot_values = get_plot_values(solution_type);
    value_range = get_value_range(range);
    solution_plot = {self.x_lims, self.y_lims, plot_values, value_range};

    function plot_values = get_plot_values(solution_type)
        if strcmp(solution_type, 'solution')
            plot_values = self.solution(plot_pts(:));
        elseif strcmp(solution_type, 'reflection')
            plot_values = self.reflection(plot_pts(:));
        elseif strcmp(solution_type, 'data')
            plot_values = -self.f(plot_pts(:));
        end
        plot_values = real(plot_values);
        plot_values = reshape(plot_values, size(plot_pts));
    end

    function value_range = get_value_range(range)
        if ischar(range)
            value_range = get_min_max();
        else
            value_range = range;
        end
    end

    function value_range = get_min_max()
        out_region = ~inpolygon(grid_x, grid_y, ...
                 real(self.bdary_pts_plot), imag(self.bdary_pts_plot));
        max_value = max(real(plot_values(out_region)));
        min_value = min(real(plot_values(out_region)));
        value_range = [min_value, max_value];
    end
end

function plot_region(self, varargin)

% Function that handles plotting of the region, poles, bdary points, bdary
% points dense, and the solution. 'poles', 'bdary pts', and 'bdary pts
% plot' take in either a boolean, string, or cell array giving plotting
% settings, eg, true, 'xk', or {'color', '#0072BD'}. 'region' and 'bounding
% box' are the same, but they are defined as polygons, so settings need to
% be given with cell arrays, and 'EdgeColor' instead of 'color' for example
%
% Inputs:
%   'solution': Takes in boolean or one of the following strings
%   1: 'solution'. Plots the full solution
%   2: 'reflection'. Only plots the reflected part of the wave
%   3: 'data'. Only plots the boundary data, ignoring the region
%   If boolean is given, it defaults to 'solution'. Defaults to 'solution'
%
%   'range': Determines the minimum and maximum colours on the colour map.
%   Takes in either 'auto' or [min, max]. Defaults to 'auto'
%
%   'region': Determines whether to plot the region or not. Defaults to 'k'
%   
%   'poles': Determines whether the poles are plotted. Defaults to '.r'
%
%   'bdary pts': Determines whether the bdary_pts are plotted. Defaults to
%   false, or 'ob' if true is given
%
%   'bdary pts plot': Determines whether the points on the boundary used
%   for plotting and computation of the max length of the lines of poles
%   are plotted. Defaults to false, or 'xk' if true is given
%
%   'bounding box': Determines whether to plot the outline of the rectangle
%   that bounds the plotting region

    % Parsing the input from the user
    [plot_solution, solution_type, range, plot_region, region_settings, ...
     plot_poles, poles_settings, plot_bdary_pts, bdary_pts_settings, ...
     plot_bdary_pts_plot, bdary_pts_plot_settings, plot_bounding_box, ...
     bounding_box_settings] = parse_inputs_plot(varargin);

    self = ensure_plot_data_is_set(self);
    plot_solution = dont_plot_non_solution(plot_solution, solution_type);
    setup_axes()

    % Plotting each of the components
    if plot_solution
        solution_plot = get_solution_plot(self, solution_type, range);
        imagesc(solution_plot{:})
        colormap 'bone'
        colorbar
    end
    if plot_region
        plot_args = [{self.region}, region_settings(:)'];
        plot(plot_args{:})
    end
    if plot_poles
        plot_args = [{self.poles}, poles_settings(:)'];
        plot(plot_args{:})
    end
    if plot_bdary_pts
        plot_args = [{self.bdary_pts}, bdary_pts_settings(:)'];
        plot(plot_args{:})
    end
    if plot_bdary_pts_plot
        plot_args = [{boundary(self.region)}, bdary_pts_plot_settings(:)'];
        plot(plot_args{:})
    end
    if plot_bounding_box
        plot_args = [self.bounding_box, bounding_box_settings(:)'];
        rectangle(plot_args{:})
    end

    function [plot_solution, solution_type, range, plot_region, ...
            region_settings, plot_poles, poles_settings, plot_bdary_pts, ...
            bdary_pts_settings, plot_bdary_pts_plot, ...
            bdary_pts_plot_settings, plot_box, bounding_box_settings]...
            = parse_inputs_plot(varargin)
        % Setting default values
        plot_solution = true;
        solution_type = 'solution';
        range = 'auto';
        plot_region = true; region_settings ...
            = {'EdgeColor', 'k', 'FaceColor', 'w', 'FaceAlpha', 1};
        plot_poles = false; poles_settings = {'.r'};
        plot_bdary_pts = false; bdary_pts_settings = {'ob'};
        plot_bdary_pts_plot = false; bdary_pts_plot_settings = {'xk'};
        plot_box = true; bounding_box_settings = {'EdgeColor', 'k'};

        % Updating values if they have been provided
        for i = 1:2:length(varargin{1})
            key = varargin{1}{i}; value = varargin{1}{i+1};
            if strcmp(key, 'solution')
                [plot_solution, solution_type] = parse_plot_solution(value);
            elseif strcmp(key, 'range')
                range = value;
            elseif strcmp(key, 'region')
                [plot_region, region_settings] ...
                    = parse_plot_input(value, region_settings);
            elseif strcmp(key, 'poles')
                [plot_poles, poles_settings] ...
                    = parse_plot_input(value, poles_settings);
            elseif strcmp(key, 'bdary pts')
                [plot_bdary_pts, bdary_pts_settings] ...
                    = parse_plot_input(value, bdary_pts_settings);
            elseif strcmp(key, 'bdary pts plot')
                [plot_bdary_pts_plot, bdary_pts_plot_settings] ...
                    = parse_plot_input(value, bdary_pts_plot_settings);
            elseif strcmp(key, 'bounding box')
                [plot_box, bounding_box_settings] ...
                    = parse_plot_input(value, bounding_box_settings);
            end
        end

        function [plot_solution, solution_type] = parse_plot_solution(value)
            if islogical(value)
                plot_solution = value;
                solution_type = 'solution';
            else
                plot_solution = true;
                solution_type = value;
            end
        end
    
        function [plot_bool, plot_setting] = parse_plot_input(input, setting)
            if iscell(input) == false
                if input == true
                    plot_bool = true; plot_setting = setting;
                elseif input == false
                    plot_bool = false; plot_setting = setting;
                else
                    plot_bool = true; plot_setting = {input};
                end
            else
                plot_bool = true; plot_setting = input;
            end
        end
    end

    function self = ensure_plot_data_is_set(self)
        if isempty(self.bounding_box)
            self = self.set_plot_data();
        end
    end

    function plot_solution = dont_plot_non_solution(...
            plot_solution, solution_type)
        if isempty(self.solution) && ~strcmp(solution_type, 'data')
            plot_solution = false;
        end
    end
    
    function setup_axes()
        clf
        hold on
        axis_limits = [self.x_lims(:)', self.y_lims(:)'];
        axis(axis_limits)
        daspect([1 1 1])
    end
end

end

end

function arr = get_arr(input, varargin)
    % Takes in either a scalar or array. If array is given, the array is
    % returned, but if a scalar is given, an array where every element of
    % that vector is returned.
    if isscalar(input)
        arr = input*ones(varargin{1}, 1);
    else
        arr = input;
    end
end
\end{lstlisting}

\subsection{Lightning Method Interface}
This is the code we use in practise to find the optimal parameters. Our code to generate almost all figures in this dissertation was based on this.
\begin{lstlisting}
warning('off', 'MATLAB:rankDeficientMatrix')

t = tiledlayout('flow');
t.Padding = 'compact';
t.TileSpacing = 'compact';

p_range = 50;
p_rate_range = 2;
s_range = 200;
A_range = 4;
B = 4;
N_range = 1;

prob = helm('size', sqrt(2));
prob = prob.set_plot_data();
prob = prob.set_problem_data();

for i = 1:length(p_range)
    for j = 1:length(p_rate_range)
        for k = 1:length(s_range)
            for l = 1:length(A_range)
                for m = 1:length(N_range)
                    p = p_range(i);
                    p_rate = p_rate_range(j);
                    s = s_range(k);
                    A = A_range(l);
                    N = N_range(m);

                    prob = prob.set_poles('poles', p,...
                        'pole exp rate', p_rate);
                    prob = prob.set_bdary_pts('bdary pts', s,...
                        'distribution', {'exponential', A, B, 1});
                    prob = prob.solve('newman', N);
                    prob.error
    
                    nexttile
                    semilogy(prob.errors)
                    title("p="+string(p)+",s="+string(s)+",rate="...
                        +string(p_rate)+",A="+string(A)+",N="+string(N))
                    %prob.plot_region()
                end
            end
        end
    end
end
\end{lstlisting}

\end{document}